\documentclass[11pt]{amsart}
\usepackage{amsmath, amssymb, amsfonts, mathtools, amscd, amsthm, color, latexsym, graphicx, tensor, array, newclude, stmaryrd, mathrsfs}
\usepackage[textwidth=20mm]{todonotes}
\usepackage[centering]{geometry}
\usepackage{tikz-cd}
\usepackage[shortlabels]{enumitem}
\usepackage{hyperref}      
\hypersetup{pdfpagelabels,
                 plainpages=false,
                    colorlinks=true,       
                    linkcolor=red,
                    citecolor=blue}

\usepackage[nameinlink]{cleveref}
\usepackage{dirtytalk}
\theoremstyle{plain}
\newtheorem{thm}{Theorem}[section]

\newtheorem{prop}[thm]{Proposition}

\newtheorem{lem}[thm]{Lemma}

\newtheorem{cor}[thm]{Corollary}

\theoremstyle{definition}
\newtheorem{que}[thm]{Question}

\newtheorem{defi}[thm]{Definition}
\newtheorem{rem}[thm]{Remark}

\newcommand{\alg}{\textnormal{alg}}

\newcommand{\an}{\textnormal{an}}

\newcommand{\N}{{\mathbb N}}
\newcommand{\Z}{{\mathbb Z}}
\newcommand{\Q}{{\mathbb Q}}
\newcommand{\R}{{\mathbb R}}
\newcommand{\C}{{\mathbb C}}

\newcommand{\cB}{{\mathcal B}}
\newcommand{\cC}{{\mathcal C}}
\newcommand{\cD}{{\mathcal D}}

\newcommand{\cF}{{\mathcal F}}

\newcommand{\cL}{{\mathcal L}}

\newcommand{\cS}{{\mathcal S}}

\newcommand{\bbN}{{\mathbb N}}

\newcommand{\bbR}{{\mathbb R}}

\DeclareMathOperator{\poly}{poly}

\DeclareMathOperator{\rpf}{rPfaff}
\DeclareMathOperator{\rpfaff}{\R_{\rpf}}
\DeclareMathOperator{\pf}{Pfaff}
\DeclareMathOperator{\pfaff}{\R_{\pf}}
\DeclareMathOperator{\semialgebraic}{\R_{\alg}}
\DeclareMathOperator{\ranalytic}{\R_{\an}}

\newcommand{\format}{{O_{\cF}(1)}}
\newcommand{\rformat}{{O_{\cF,r}(1)}}
\newcommand{\degree}{{\poly_{\cF}(D)}}
\newcommand{\rdegree}{{\poly_{\cF,r}(D)}}
\newcommand{\grdegree}{{\poly_{\cF}(D,r)}}

\newcommand{\pderivative}[2]{\frac{\partial#1}{\partial#2}}
\def\so{\raisebox{.5ex}{\scalebox{0.6}{\#}}\kern-.02em{}o}
\def\s{\raisebox{.5ex}{\scalebox{0.6}{\#}}}
\newcommand\bigsupset[1][1.7]{%
   \mathrel{\vcenter{\hbox{\scalebox{#1}{$\supset$}}}}}

\title[Sharply o-minimal structures and sharp cellular decomposition]{Sharply o-minimal structures and sharp cellular decomposition}
\author{Gal Binyamini, Dmitry Novikov, Benny Zak}
\date{\today}
\begin{document}

\begin{abstract}
Sharply o-minimal structures (denoted \so-minimal) are a strict subclass of the o-minimal structures, aimed at capturing some finer features of structures arising from algebraic geometry and Hodge theory. Sharp o-minimality associates to each definable set a pair of integers known as \emph{format} and \emph{degree}, similar to the ambient dimension and degree in the algebraic case; gives bounds on the growth of these quantities under the logical operations; and allows one to control the geometric complexity of a set in terms of its format and degree. These axioms have significant implications on arithmetic properties of definable sets -- for example, \so-minimality was recently used by the authors to settle Wilkie's conjecture on rational points in $\R_{\textnormal{exp}}$-definable sets.

In this paper, we develop some basic theory of sharply o-minimal structures. We introduce the notions of reduction and equivalence on the class of \so-minimal structures. We give three variants of the definition of \so-minimality, of increasing strength, and show that they all agree up to reduction. We also consider the problem of \say{sharp cellular decomposition}, i.e., cellular decomposition with good control on the number of cells and their formats and degrees. We show that every \so-minimal structure can be reduced to one admitting sharp cellular decomposition. We use this to prove bounds on the Betti numbers of definable sets in terms of format and degree. 
\end{abstract}
\maketitle

\section{Introduction}
\subsection{Notation}
For positive integers $m\geq n$, let $\pi^{m}_{n}:\R^{m}\to\R^{n}$ be the standard projection to the first $n$ coordinates. Usually, $m$ is clear from the context and is omitted from the notation. We sometimes restrict $\pi^{m}_{n}$ to $I^{m}$ where $I:=(0,1)$, without changing the notation. 
For a set $X\subset\R^{\ell}$ we denote $X^{c}$ to be the complement of $X$ and $\partial X$ to be the \emph{frontier} of $X$, that is, $\partial X:=\overline{X}\setminus X$. Finally, we use the inclusion symbol $\subset$ to denote \emph{weak} inclusion. To indicate a strict inclusion, we use the symbol $\subsetneq$. \\ 

\subsection{About this paper}
Sharply o-minimal structures (see \Cref{sec:Sharp_structures}) are o-minimal structures with a double filtration on the collection of all definable sets by two natural numbers, \say{degree} and \say{format} (called FD-filtration) satisfying some natural axioms. They were introduced by the first two authors in \cite{SharpIMC}, and used by the authors in \cite{Wilkie_conj_Proof} to prove Wilkie's conjecture \cite[Conjecture 1.11]{Pila-Wilkie}. The axioms put forth in \cite{Wilkie_conj_Proof} and \cite{SharpIMC} can be naturally weakened or strengthened, and one of the main goals of the present paper is to show that such changes in the axioms will result in essentially equivalent classes of o-minimal structures. So, o-minimal structures (together with an FD-filtration) satisfying the axioms introduced in \cite{Wilkie_conj_Proof},\cite{SharpIMC} are henceforth called \emph{Weakly} sharply o-minimal, or W\so-minimal for short. The name \say{Sharply o-minimal structures} (or \so-minimal for short) will now refer to o-minimal structures (together with an FD-filtration) satisfying a stronger set of axioms, in particular \so-minimal structures are also W\so-minimal. We also introduce presharp o-minimal structure, (or P\so-minimal for short), which are o-minimal structures (again, together with an FD-filtration) satisfying a set of axioms weaker than those put forth in \cite{SharpIMC},\cite{Wilkie_conj_Proof}. In particular, W\so-minimal structures are also P\so-minimal. See \Cref{defi:psharp_o_min}, \Cref{defi:wsharp_o_min}, \Cref{defi:sharp_o_min} for the formal introduction of these axioms, as well as \Cref{defi:fd_filtrations} for the formal definition of FD-filtrations. \\ 

One of the cornerstones of the theory of o-minimality is the notion of cellular decomposition and the theorem about its existence. We refer the unfamiliar reader to \cite{van_den_dries_book} or to \Cref{sec:Sharp_cellular_decomposition} below for the notions of cells and cellular decompositions. Roughly speaking, the cellular decomposition theorem asserts that any definable set can be decomposed into a finite number of \say{simple} sets called cells. Moreover, the number of cells required is uniform in families. In either of the sharp, presharp, or weakly sharp settings, where sets are endowed with natural numbers $\cF,D$ (format and degree) that govern their complexity, it is natural to ask if the number of cells needed, as well as their own format and degree, can be controlled in terms of $\cF,D$.
\begin{que}
    Can one find sequences $\{P_{\cF}\}_{\cF\in\N}\subset\N[D,k],\{Q_{\cF}\}_{\cF\in\N}\subset\N[D]$ of polynomials and a function $C:\N\to N$ such that the following holds. Given any collection $X_{1},\dots,X_{k}\subset\R^{\ell}$ of definable sets, all of which have format $\cF$ and degree $D$, it is possible to decompose $\R^{\ell}$ into $P_{\cF}(D,k)$ cells that are compatible with $X_{1},\dots,X_{k}$, such that the cells have format $C(\cF)$ and degree $Q_{\cF}(D)$?
\end{que}
 If the answer is yes, we say the structure has sharp cellular decomposition, or \s CD for short. Just as the classical cellular decomposition theorem is crucial for applications of o-minimality, \s CD is crucial for applications of \so-minimality. For instance, all of the results of \cite{Wilkie_conj_Proof} are proved under the assumption of \s CD. Moreover, in this paper, we show that under the assumption of \s CD, all three of the introduced axiom sets (those that define P\so-minimality, W\so-minimality, \so-minimality) are equivalent. \\

It is not currently known whether every sharply o-minimal structure has \s CD or not. In this paper, we prove a weaker result sufficient for quantitative applications. Namely, given a P\so-minimal structure $(\cS,\Omega)$ (here $\cS$ is the structure and $\Omega$ is an FD-filtration on $\cS$) we define a new FD-filtration $\Omega^{'}$ of $\cS$ such that $(\cS,\Omega^{'})$ is a \so-minimal structure with \s CD, and such that $\Omega$ is \emph{reducible} to $\Omega'$. This means that if a set $X$ has format $\cF$ and degree $D$ with respect to the FD-filtration $\Omega$, then it has format $C(\cF)$ and degree $P_{\cF}(D)$ with respect to the new FD-filtration $\Omega'$, where $C:\N\to\N$ is a fixed function and $\{P_{\cF}\}_{\cF\in\N}\subset\N[D]$ is a fixed sequence of polynomials (see \Cref{defi:reduction_of_filtration} below for the formal definition).
This construction is inspired by the one developed in \cite{Pfaffian_cells}, where the authors produce an FD-filtration $\Omega^{*}$ on $\rpfaff$ (the structure generated by restricted Pfaffian functions), which is weakly-sharp with \s CD, out of an FD-filtration $\Omega$ introduced by Gabrielov and Vorobjov (see \cite{Pfaffaian_complexity} or \cref{subsec:Pffafian_example}). It is not known whether $\left(\rpfaff,\Omega\right)$ is W\so-minimal or if it has \s CD. The FD-filtration $\Omega^{*}$ on the structure $\rpfaff$ was the one used in \cite{Wilkie_conj_Proof}. 

 As an application of our main results, in the last section of this paper we discuss sharp triangulations of sets definable in sharply o-minimal structures, and prove a bound on their Betti numbers. 
\subsection{Sharply o-minimal structures}
\label{sec:Sharp_structures}
Let $\cS$ be any structure, though in this paper we will always work with an o-minimal expansion of the real field. We introduce the notion of an FD-filtration, a filtration on the collection of all definable sets in $\cS$ by two natural numbers, called format (denoted $\cF$) and degree (denoted $D$).

\begin{defi}[FD-filtrations]
\label{defi:fd_filtrations}
We say that $\Omega=\{\Omega_{\cF,D}\}_{\cF,D\in\N}$ is an FD-filtration on  $\cS$ if 
\begin{enumerate}
    \item every $\Omega_{\cF,D}$ is a collection of definable sets,
    \item $\Omega_{\cF,D}\subset\Omega_{\cF+1,D}\cap\Omega_{\cF,D+1}$ for every $\cF,D$, and
    \item every definable set belongs to $\Omega_{\cF,D}$ for some $\cF,D$.
\end{enumerate}  
If $X\in\Omega_{\cF,D}$ we say that $X$ has format $\cF$ and degree $D$. We  say that a definable function $f:X\to Y$ has format $\cF$ and degree $D$ if its graph $\Gamma_{f}$ is in $\Omega_{\cF,D}$. We will also abuse notation by writing $f\in\Omega_{\cF, D}$, and generally by referring to $f$ as 'being in' $\Omega_{\cF, D}$, to indicate that $f$ has format $\cF$ and degree $D$.
\end{defi}
\begin{rem}
    Note that FD-filtrations can be naturally intersected. If $\Omega^{1},\Omega^{2}$ are two FD-filtrations, then we can define $\Omega^{3}_{\cF,D}:=\Omega^{1}_{\cF,D}\cap\Omega^{2}_{\cF,D}$ for every pair $(\cF,D)$ of natural numbers, and then the collection $\{\Omega^{3}_{\cF,D}\}_{\cF,D}$ is again an FD-filtration, denoted by $\Omega^{1}\cap\Omega^{2}$. 
\end{rem}
We introduce the notion of reduction and equivalence of FD-filtrations.
\begin{defi}[Reduction of FD-filtrations]
\label{defi:reduction_of_filtration}
  Let $\Omega,\Omega'$ be two FD-filtrations on a structure $\cS$. We say
  that $\Omega$ is \emph{reducible} to $\Omega'$ and write
  $\Omega\le\Omega'$ if there exists a function $a:\N\to\N$, and for every $\cF\in\N$  a non-zero polynomial $P_{\cF}\in\N[D]$ such that
  \begin{equation}
    \Omega_{\cF,D}\subset \Omega'_{a(\cF),P_{\cF}(D)} \qquad \forall \cF,D\in\N.
  \end{equation}
  In other words, $\Omega\le\Omega'$ if $\Omega_{\cF,D}\subset\Omega^{'}_{\format,\degree}$ for every $\cF,D$.
  We say that $\Omega,\Omega'$ are equivalent if $\Omega\le\Omega'$ and $\Omega'\le\Omega$, and write $\Omega\simeq\Omega'$.
\end{defi}
\begin{rem}
\label{rem:star_counting}
The following is the main idea of the definition of reduction. Let $\cS$ be an o-minimal expansion of the real field, and let $\Omega\leq\Omega'$ be two FD-filtrations on $\cS$. Let
\begin{equation}
    F(X_1,\dots,X_{m},f_{1},...,f_{n})
\end{equation} 
be a function that associates a positive real number to every tuple of $m$ definable sets and $n$ definable functions. Suppose that $\Omega'$ possesses the following property: there exist functions $a'(\cF),b'(\cF)>0$, such that if the sets $X_{j}$ and functions $f_{i}$ are in $\Omega'_{\cF,D}$, then $F(X_1,\dots,X_{m},f_{1},\dots,f_{m})\leq a'(\cF)D^{b'(\cF)}$. Then there exist functions $a(\cF),b(\cF)>0$ such that the analogous statement is true for $\Omega$, that is, such that if the sets $X_{j}$ and functions $f_{i}$ are in $\Omega_{\cF,D}$, then $F(X_1,\dots,X_{m},f_{1},\dots,f_{m})\leq a(\cF)D^{b(\cF)}$. Thus, while there is a \say{cost} to reducing a filtration, the cost is polynomial in the degree and so keeps the quantitative aspect of \so-minimality within the same \say{scale}.
\end{rem}
We will now define \so-minimal (resp. W\so-minimal, P\so-minimal) structures. While all three variants share common axioms, we will explicitly state all of the axioms in every case for clarity. 
\begin{defi}[presharp structures]
\label{defi:psharp_o_min}
A \emph{presharp} o-minimal structure, denoted by P\so-minimal, is a pair $\Sigma=(\cS,\Omega)$ where $\cS$ is an o-minimal expansion of the real field, and $\Omega=\{\Omega_{\cF,D}\}_{\cF,D\in\bbN}$ is an FD-filtration, such that for every $\cF\in\bbN$ there exists a non-zero polynomial $P_{\cF}\in\N[D]$ such that the following axioms are satisfied.\\ \\
If $A\in\Omega_{\cF,D}$ then:
\begin{enumerate}
    \item[(P1)] if $A\subset\bbR$, it has at most $P_{\cF}(D)$ connected components,
    \item[(P2)] if $A\subset\bbR^{\ell}$ then $\cF\geq\ell$,
    \item[(P3)] if $A\subset\bbR^{\ell}$ then $\pi_{\ell-1}(A),A^{c},A\times\bbR,\bbR\times A$ are in $\Omega_{\cF+1,D}$.
 \end{enumerate}
 If $A_1,A_{2}\subset\bbR^{\ell}$ with $A_{i}\in\Omega_{\cF_i,D_i}$ and     $\cF:=\max\{\cF_{1},\cF_{2}\}$, $D:=D_1+D_2$,  then:
 \begin{enumerate}
        \item[(P4)] $A_1\cup A_2\in\Omega_{\cF+1,D}$, 
        \item[(P5)] $A_1\cap A_2\in\Omega_{\cF+1,D}$.
\end{enumerate}    
If $P\in\bbR[x_1,\dots,x_{\ell}]$, then: 
\begin{enumerate}
    \item[(P6)] $\{P=0\}\in\Omega_{\ell,\deg P}$.
\end{enumerate}
\end{defi}
\begin{defi}[weakly-sharp structures]
\label{defi:wsharp_o_min}
A \emph{weakly-sharp} o-minimal structure, denoted W\so-minimal, is a pair $\Sigma=(\cS,\Omega)$ and a sequence of polynomials $P_{\cF}$ as in \Cref{defi:sharp_o_min}, such that the following axioms are satisfied.\\ \\
If $A\in \Omega_{\cF,D}$ then:
\begin{enumerate}
    \item[(W1)] if $A\subset\bbR$, it has at most $P_{\cF}(D)$ connected components,
    \item[(W2)] if $A\subset\bbR^{\ell}$ then $\cF\geq\ell$,
    \item[(W3)] if $A\subset\bbR^{\ell}$ then $\pi_{\ell-1}(A),A^{c},A\times\bbR,\bbR\times A$ are in $\Omega_{\cF+1,D}$.
\end{enumerate}
If $A_1,\dots,A_{k}\subset\bbR^{\ell}$ with $A_{i}\in\Omega_{\cF_i,D_i}$ and $\cF:=\max_{i}\{\cF_{i}\},\;D:=\sum_{i}D_i$, then:
\begin{enumerate}
    \item[(W4)]  $\cup_{i}A_{i}\in\Omega_{\cF,D}$,    
    \item[(W5)]  $\cap_{i}A_{i}\in\Omega_{\cF+1,D}$.
\end{enumerate}
If $P\in\bbR[x_1,\dots,x_{\ell}]$, then:
\begin{enumerate}
    \item[(W6)] $\{P=0\}\in\Omega_{\ell,\deg P}$.
\end{enumerate}
\end{defi}
\begin{rem}[comparison between P\so-minimal and W\so-minimal structures]
The only differences are in axioms 4,5. In weakly-sharp structures, taking finite unions no longer increases format. Moreover, now arbitrary (finite) intersections increase the format only by $1$. Note that W\so-minimal structures are also presharp. 
\end{rem}
\begin{defi}[sharp structures]
\label{defi:sharp_o_min}
A \emph{sharp} o-minimal structure or \emph{sharply o-minimal} structure, denoted \so-minimal, is a pair $\Sigma=(\cS,\Omega)$ and a sequence of polynomials $P_{\cF}$ as in \Cref{defi:sharp_o_min}, such that the following axioms are satisfied.\\ \\
If $A\in \Omega_{\cF,D}$ then:
\begin{enumerate}
    \item[(\s1)] if $A\subset\bbR$, it has at most $P_{\cF}(D)$ connected components,
    \item[(\s2)] if $A\subset\bbR^{\ell}$ then $\cF\geq\ell$,
    \item[(\s3)] if $A\subset\bbR^{\ell}$ then $\pi_{\ell-1}(A),A^{c}$ are in $\Omega_{\cF,D}$, while $A\times\bbR,\bbR\times A$ are in $\Omega_{\cF+1, D}$.
\end{enumerate}
If $A_1,\dots,A_{k}\subset\bbR^{\ell}$ with $A_{i}\in\Omega_{\cF_i,D_i}$ and $\cF:=\max_{i}\{\cF_{i}\},\;D:=\sum_{i}D_i$, then:
\begin{enumerate}
    \item[(\s4)]  $\cup_{i}A_{i}\in\Omega_{\cF,D}$,    
    \item[(\s5)]  $\cap_{i}A_{i}\in\Omega_{\cF,D}$.
\end{enumerate}
If $P\in\bbR[x_1,\dots,x_{\ell}]$, then:
\begin{enumerate}
    \item[(\s6)] $\{P=0\}\in\Omega_{\ell,\deg P}$.
\end{enumerate}
\end{defi}
\begin{rem}[comparison between W\so-minimal and \so-minimal structures]
There are differences in the third and fifth axioms. Firstly, taking projections and complements no longer increases the format in \so-minimal structures.  Note crucially that multiplying by $\R$, i.e, \say{adding new variables} has to increase the format in all versions of the axioms. Moreover, in sharp structures, taking finite intersections no longer increases the format. Note that \so-minimal structures are in particular weakly-sharp. 
\end{rem}
\begin{rem}
Given an o-minimal expansion $\cS$ of the real field, we will sometimes say that an FD-filtration is sharp (resp. weakly-sharp, presharp) if the pair $(\cS,\Omega)$ is sharp (resp. weakly-sharp, presharp) for brevity. More generally, we sometimes refer to properties of $(\cS,\Omega)$ as being properties of $\Omega$, if $\cS$ is fixed and understood from context. 
\end{rem}
\begin{rem}
We recall once again that in \cite{Wilkie_conj_Proof,SharpIMC}, W\so-minimal structures were called \so-minimal structures.
\end{rem}
\begin{rem}
    \label{rem:asymptotic_notation}
    Let us introduce a useful, though non-standard piece of notation. The symbol $O_{a}(1)$ denotes a specific universally fixed positive constant (possibly different at each occurrence) $C(a)$ that depends only on $a$. The symbol $\poly_{a}(b)$ denotes a universally fixed polynomial $P_{a}(b)$ in $b$ with positive coefficients that depends only on $a$. Thus, rather than representing (asymptotic) classes of functions like ordinary asymptotic notations, these symbols are simply stand-ins for specific constants and polynomials we do not keep track of. Almost always, $a$ is a natural number denoted $\cF$ and called \emph{format} while $b$ is a natural number denoted $D$ and called \emph{degree}. Thus, for instance, instead of saying \say{There exists a polynomial $P_{\cF}\in\N[D]$ such that if $X$ is a set of format $\cF$ and degree $D$ then $X$ has at most $P_{\cF}(D)$ connected components}, we may say \say{If $X$ has format $\cF$ and degree $D$, then it has $\degree$ connected components}.
\end{rem}
Strictly speaking, we need a bound on the number of connected components of a definable set in arbitrary dimension. Due to the following proposition, it is sufficient to axiomatize this bound for subsets of $\R$ as in (\s 1),(P1),(W1). See \Cref{sec:bound_on_H_0} for the proof, which requires no further terminology.
\begin{prop}
\label{prop:bound_on_H_0} 
Fix a presharp structure, and let $X\subset\bbR^{\ell}$ be a definable set of format $\cF$ and degree $D$. Then $X$ has at most $\degree$ connected components.
\end{prop}
\begin{rem}
    While we can bound the number of connected components, in general, we cannot estimate the formats and degrees of connected components, and in fact, this proves to be the one \say{obstruction} to proving \s CD in general. See \Cref{sec:Sharp_cellular_decomposition} for more details. 
\end{rem}
We now state a simplified version of our main result, which is largely inspired by \cite{Pfaffian_cells}.
\begin{thm}
\label{thm:for_show}
Let $(\cS,\Omega)$ be a presharp o-minimal structure. Then there exists an FD-filtration $\Omega'$ such that $\Omega\le\Omega'$, and moreover $(\cS,\Omega')$ is a \so-minimal structure with \s CD. 
\end{thm}
\begin{rem}
    In fact, the filtration $\Omega'$ mentioned in \Cref{thm:for_show} is equivalent to a key FD-filtration $\Omega^{*}$ that can be explicitly defined given $\Omega$. See \Cref{sec:Sharp_cellular_decomposition} below.
\end{rem}
This theorem is proved in two key steps. The first is to construct a new filtration $\Omega''$ such that $\Omega$ is reducible to $\Omega''$ and such that $(\cS,\Omega'')$ is W\so-minimal with \s CD. The following proposition gives the second step, the proof of which can be found in \Cref{sec:main_proof}.
\begin{prop}
\label{prop:Wsharp_SCD_isSharp}
 Let $(\cS,\Omega)$ be a W\so-minimal structure with \s CD. Then there exists a filtration $\Omega'$ that is equivalent to $\Omega$, such that $(\cS,\Omega')$ is \so-minimal with \s CD.
\end{prop}
So, applying \Cref{prop:Wsharp_SCD_isSharp} to $(\cS,\Omega'')$ would finish the proof of \Cref{thm:for_show}.
Thus, the current state of affairs can be summarized in the following diagram, where the arrow represents reduction, and the equality represents equivalence. 
\begin{equation*}
\begin{tikzcd}
P\s\textnormal{o-minimal}\arrow[r, phantom, sloped, "\bigsupset"]  \arrow[d, phantom, sloped, "\bigsupset"]\arrow{dr} & W\s\textnormal{o-minimal}\arrow[r, phantom, sloped, "\bigsupset"] \arrow[d, phantom, sloped, "\bigsupset"] & \s\textnormal{o-minimal}\arrow[d, phantom, sloped, "\bigsupset"] \\
P\s\textnormal{o-minimal}+\s\textnormal{CD} \arrow[r, phantom, sloped, "\bigsupset"] & W\s\textnormal{o-minimal}+\s\textnormal{CD} \rar[equal] & \s\textnormal{o-minimal}+ \s\textnormal{CD}
\end{tikzcd}
\end{equation*}
\subsection{Generated FD-filtrations}
It is often convenient to consider a structure $\cS$ \emph{generated} by a collection of sets $\{A_{\alpha}\subset\R^{\ell_{\alpha}}\}_\alpha$, that is, $\cS$ is the minimal structure in which all of the sets $A_{\alpha}$ are definable. For example, the semi-algebraic structure $\R_{\alg}$ is generated by the collection of (real) Zariski-closed sets. It is o-minimal due to a celebrated result of Tarski and Seidenberg, see \cite{tarski_seidenberg}, for example. Similarly, $\R_{\textnormal{exp}}$ is the structure generated by all the sets in $\R_{\alg}$, and the graph of the exponential function defined on all of $\R$. This structure is o-minimal, and this is a highly non-trivial fact, that follows from e.g the work of Wilkie, see \cite{wilkie:complements}. 

Similarly, one can consider FD-filtrations sharply (resp. presharply, weakly-sharply) generated by some initial associations of formats and degrees for some collection of definable sets. By this we mean the \emph{minimal} FD-filtration 'containing' the initial data and satisfying axioms (\s2)-(\s6) (resp. axioms (P2)-(P6) and (W2)-(W6)). To this end, we introduce a partial order on the collection of FD-filtrations (on a fixed structure).
\begin{defi}
    Let $\cS$ be any structure, and let $\Omega^{1},\Omega^{2}$ be two FD-filtrations on $\cS$. We say that $\Omega^{2}$ \emph{refines} $\Omega^{1}$, denoted $\Omega^{1}\subset\Omega^{2}$, if for every pair $(\cF,D)\in\N^{2}$ we have the inclusion $\Omega^{1}_{\cF,D}\subset\Omega^{2}_{\cF,D}$.
\end{defi}
\begin{rem}
    Note that the relation $\subset$ of refinement is a partial order on the collection of FD-filtrations on $\cS$, while the previously introduced relation $\leq$ of reduction (see \Cref{defi:reduction_of_filtration}) is not.
\end{rem}
\begin{lem}
\label{lem:generated_filtrations}
    Let $\cS$ be an o-minimal structure that is generated by a collection $\{A_{\alpha}\subset\R^{\ell_{\alpha}}\}_{\alpha}$ of definable sets. Suppose that for every $\alpha$ there is an associated set $\cD_{\alpha}=\{(\cF_{\alpha\beta},D_{\alpha\beta})\}_{\beta}$ of pairs of natural numbers satisfying $\cF_{\alpha\beta}\geq\ell_{\alpha}$ for every $\alpha,\beta$. Then there exists a unique FD-filtration $\Omega$  (resp. $\Omega^{\textnormal{pre}},\Omega^{\textnormal{weak}}$) on $\cS$ that satisfies axioms (\s2)-(\s6) (resp. (P2)-(P6) and (W2)-(W6)) with the following properties.\begin{enumerate}
        \item For every $\alpha,\beta$ the set $A_{\alpha}$ is in $\Omega^{\textnormal{pre}}_{\cF_{\alpha\beta},D_{\alpha\beta}}$, in $\Omega^{\textnormal{weak}}_{\cF_{\alpha\beta},D_{\alpha\beta}}$ and in $\Omega_{\cF_{\alpha\beta},D_{\alpha\beta}}$.
        \item The FD-filtration $\Omega$ (resp. $\Omega^{\textnormal{pre}},\Omega^{\textnormal{weak}}$) is minimum among all FD-filtrations with property (1) above, satisfying the axioms (\s2)-(\s6) (resp. (P2)-(P6), (W2)-(W6)).  
    \end{enumerate}
\end{lem}
The proof of this lemma is left as an instructive exercise for the reader. Indeed, it is not hard to see that such FD-filtrations exist, and their intersection will be the minimum among all such FD-filtrations. One can construct these filtrations by induction on format, or one can use a suitable variation of the notion of \emph{structure trees}, introduced in \Cref{subsection:Structure_Trees}, see \Cref{lem:ST_generates_filtrations}. We call $\Omega,\Omega^{\textnormal{pre}},\Omega^{\textnormal{weak}}$ the FD-filtration \emph{sharply} (pre-sharply, weakly-sharply respectively) generated by the datum $\left\{\mathcal{D}_{\alpha}\right\}_{\alpha}$.
\begin{rem}
   It is often convenient to take all definable sets as a generating collection and use the degrees and formats that a fixed FD-filtration gives. So, for example, one can consider the FD-filtration sharply generated by a given presharp FD-filtration.  
\end{rem}
\begin{rem}
    Crucially, unlike what the name suggests, note that $(\cS, \Omega)$ (resp. $(\cS,\Omega^{\textnormal{pre}})$, $(\cS,\Omega^{\textnormal{weak}})$) is not automatically \so-minimal (resp. P\so-minimal, W\so-minimal). Rather, it automatically satisfies axioms $(\s 2)-(\s 6)$ (resp. axioms (P2)-(P6), (W2)-(W6)), and is \so-minimal if and only if in addition it satisfies axiom (\s1) (resp. axiom (P1), (W1)). 
\end{rem}
\subsection{Examples and non-examples of \so-minimal structures}
We briefly recall below the only two known examples of \so-minimal structures (the semialgeraic and Pfaffian structures), and explain why $\ranalytic$ cannot be made into a \so-minimal structure. We refer the reader to \cite{SharpIMC} for further details. 
\subsubsection{The semi-algebraic structure $\semialgebraic$}
Let $\Omega'_{\ell,D}$ be the collection of sets $X$ in $\R^{\ell}$ which are presentable as a union of finitely many basic sets
\begin{equation}
    \{P_1=\dots=P_k=0, \;Q_1,\dots,Q_s>0\}, \quad P_i, Q_j\in\mathbb{R}[x_{1},\dots,x_\ell],
\end{equation}
such that the sum of the degrees of all $P$ and $Q$ over all these basic sets defining $X$ is $D$. The filtration $\Omega'$ is not \so-minimal. However, the filtration $\Omega$ sharply generated by $\Omega'$ is \so-minimal. This is a non-trivial fact that follows from effective cellular decomposition in the semialgebraic category, see \cite[Chapter 5]{Algos_real_algebrometry}.
\subsubsection{The analytic structure $\R_\an$}
Not surprisingly, $\ranalytic$ is not sharply o-minimal with respect to
any FD-filtration. In fact, it is not even presharp, but by \Cref{thm:for_show} it is sufficient to show that it is not sharp. Assume the contrary. Let $\omega_1=1$ and
$\omega_{n+1}=2^{\omega_n}$, and let $\Gamma=\{y=f(z)\}\subset\C^2$
denote the graph of the holomorphic function
$f(z)=\sum_{j=1}^\infty z^{\omega_j}$ restricted to the closed disc of radius
$1/2$ centered at the origin. Clearly $\Gamma$ is definable in $\ranalytic$, and by the axioms of sharpness, the set
\begin{equation}
  X_{\epsilon,n}:=\Gamma\cap\big\{y=\epsilon+\sum_{j=1}^n z^{\omega_j}\big\},
\end{equation}
has a universally bounded format (that depends only on the format of $\Gamma$), and its degree is the maximum among $\omega_{n}$ and the degree of $\Gamma$, which for large $n$ is just $\omega_{n}$. Therefore, the
number of points in $X_{\epsilon,n}$
is bounded by $P(\omega_n)$, where $P$ is some fixed polynomial with positive coefficients that depends only on $\Gamma$ (and not on $\epsilon$). Note that $X_{\epsilon,n}$ is not a subset of $\R$, but the desired bound on its cardinality still follows from e.g \Cref{prop:bound_on_H_0}, or from \Cref{thm:for_show}. Fix an $N$ large enough such that $\omega_{N+1}=2^{\omega_{N}}>P(\omega_{N})$. By basic complex analysis, if $\epsilon$ is small enough, the number of points in $X_{\epsilon,N}$ is at least 
$\omega_{N+1}$, which is a contradiction.
\subsubsection{Pfaffian structures}
\label{subsec:Pffafian_example}
\label{marker}
Let $B\subset\R^\ell$ be an open box. A sequence
$f_1,\ldots,f_m:B\to\R$ of real-analytic functions is called a
\emph{Pfaffian chain} if they satisfy a triangular system of algebraic
differential equations of the form
\begin{equation}
  \pderivative{f_i}{x_j} = P_{ij}(x_1,\ldots,x_\ell,f_1,\ldots,f_i), \qquad \forall 1\le i\le m,1\le j\le \ell,
\end{equation}
where $P_{ij}$ are polynomials with real coefficients. 
The Pfaffian chain is called \emph{restricted} if $B$ is bounded and
$f_1,\ldots,f_m$ extend as real analytic functions to a neighborhood of $\bar B$. A
Pfaffian function $f$ is a polynomial in the variables and the functions from the chain, i.e a function of the form
$Q(x_1,\ldots,x_\ell,f_1,\ldots,f_m)$ where $Q$ is a polynomial with real coefficients. The degree of $f$ is defined to be the degree of $Q$. We denote the structure
generated by the Pfaffian functions by $\pfaff$, and its restricted
analog by $\rpfaff$.

Gabrielov and Vorobjov \cite{Pfaffaian_complexity} defined an FD-filtration $\Omega$ on $\rpfaff$, which is not known to be sharp (or even presharp). Roughly speaking, in $\Omega$, the format of a semipfaffian set $X\subset\R^{k}$ is the maximum among $k$ and the length of the Pfaffian chains defining the Pfaffian functions appearing in a representation of $X$ as a finite union of basic sets, and the degree of $X$ is the sum of the degrees of all Pfaffian functions and the polynomials $P_{ij}$ in all Pfaffian chains appearing in the same representation. If $Y$ is a projection of $X$, the format and degree of the subpfaffian set $Y$ are defined to be those of $X$. While Gabrielov and Vorobjov were able to obtain bounds on the sum of the Betti numbers of a semipfaffian set $X$ which are polynomial in the degree of $X$, they were not able to obtain the same bounds for subpfaffian sets in full generality. The main, and crucial, difficulty is that if $A\in\Omega_{\cF,D}$ then it is only known that $A^{c}\in\Omega_{\degree,\degree}$, i.e. the format of $A^c$ depends also on the degree of $A$.\\ 

In \cite{Pfaffian_cells}, Binyamini and Vorobjov introduce a new notion of degree for subpfaffian sets, with which they do achieve polynomial bounds on the Betti numbers of subpfaffian sets. Essentially, they introduce an FD-filtration $\Omega^{*}$ based on $\Omega$ that  makes $\rpfaff$ into a W\so-minimal structure with \s CD. As $\Omega\le\Omega^{*}$, they obtain bounds on the sum of the Betti numbers of subpfaffian sets which are polynomial in the degree in the sense of \cite{Pfaffaian_complexity}.  As mentioned above, our construction generalizes the construction of $\Omega^{*}$ from \cite{Pfaffian_cells} to the settings of presharp o-minimal structures. 

\subsection{Format and degree of first order formulae}
Let $(\cS,\Omega)$ be a \so-minimal structure, and let $\cL$ be the language with atomic predicates of the form $(x\in X)$ for every definable set $X$, and with neither constants nor function symbols. We will assume that the variables of $\cL$ are linearly ordered, so if a formula $\psi$ has $n$ free variables, then it uniquely defines a set in $\R^{n}$. The goal of this section is to filter the formulae in $\cL$ by format and degree, such that if a formula $\psi$ has format $\cF$ and degree $D$ then it defines a set in $\Omega_{\format,\degree}$, see \Cref{prop:fordeg_of_formula_set}. Note that this has already been done in \cite{Pfaffian_cells} for the structure $\rpfaff$.  
\begin{defi}[Format and degree of formulae]
\label{defi:fordeg_of_sformula}
Let $\psi$ be a formula. Suppose that there are $n$ different variables appearing in $\psi$ (either free or quantified), and let $X_{j}\in\Omega_{\cF_{j},D_{j}}$ be the sets appearing in the atomic predicates of $\psi$. Denote $\cF=\max{\cF_{j}}$ and $D:=\sum D_{j}$. Then we say that $\psi$ has format $\max\{\cF,n\}$ and degree $D$. 
\end{defi}
In this text, we will also need a notion of P-format, and it is natural to consider a notion of W-format as well. Since geometric operations (that correspond to first order operations on $\psi$) affect format and degree differently in all three variants of sharpness, the definition of format and degree for formulae should depend on the type of structure (presharp, weakly sharp, or sharp).  Moreover, since in presharp or weakly-sharp case geometric operations such as intersections and projections may increase the format, unlike the format of a formula $\psi$, the P-format and W-format of $\psi$ can't be defined just in terms of its atoms with \Cref{prop:fordeg_of_formula_set} in mind. Rather, the P-format and W-format will depend on the binary parse-tree of $\psi$. 
\begin{defi}[P-format]
\label{defi:fordeg_of_formula}
Let $\psi$ be a formula, and $d$ be the depth of the (unique, binary) parse-tree of $\psi$. Then the P-format of $\psi$ is defined to be $\max\{\cF,d\}$, where $\cF$ is the format of $\psi$.
\end{defi}
For the definition of W-format, one needs to consider a different kind of parse-tree (one which is not necessarily binary), and moreover disregard vertices of the tree that are associated with disjunction. Since we won't actually need the W\so-minimal case of the following proposition in this paper, we omit a formal definition of W-format. The following proposition is clear from the definition. In fact, axioms (\s1) or (P1) are not even needed. 
\begin{prop}
\label{prop:fordeg_of_formula_set}
Let $(\cS,\Omega)$ be a \so-minimal structure (resp. P\so-minimal), and let $\psi$ be a formula of format (resp. P-format) $\cF$ and degree $D$. Then the set that $\psi$ defines is in $\Omega_{\format,\degree}$.
\end{prop}
We remark that an analogous statement holds for the weakly-sharp case. \\ 
To illustrate the mechanism of the axioms of sharpness, and how to calculate the format and degree of a formula, we provide a detailed proof of the following lemma, which is needed later.
\begin{lem}
    \label{lem:frontier}
    Fix a presharp structure, and let $X\subset\R^{n}$ be definable of format $\cF$ and degree $D$. Then the frontier $\partial X=\overline{X}\setminus X$ has format $\format$ and degree $\degree$.
\end{lem}
\begin{proof}
    The closure of $X$ is defined by the formula \begin{equation}
        \forall r \left[(r>0)\rightarrow\exists x (x\in X)\wedge |y-x|<r\right].
    \end{equation}
This formula has $n$ free variables (the coordinates of $y$), and its atomic predicates are the formulas associated to the set $X$, the set $\{(y,x,r)|\;|y-x|<r\}\subset\R^{n}\times\R^{n}\times\R$ and the set $\{r|\;r>0\}\subset\R$. Since $\cF\geq n$, it is clear by the axioms of presharp structures that the latter two sets have format $\format$ and $O(1)$ respectively, and both have degree $O(1)$ (note that they are semialgebraic). Overall, we conclude that this formula has format $O_{\cF}(1)$ and degree $D+O(1)=\poly_{\cF}(D)$. \\ \\ The depth of the parse tree for this formula is bounded by the number of logical quantifiers and connectives appearing in it, which is clearly $O(1)$. So this formula has $P$-format $\format$ and degree $\degree$. Thus by \Cref{prop:fordeg_of_formula_set} the closure $\overline{X}$ has format $\format$ and degree $\degree$, and so again since $\overline{X}\setminus X=\overline{X}\cap\left(\R^{n}\setminus X\right)$, by the axioms of presharp structures the frontier $\partial X$ has format $\format$ and degree $\degree$. 
\end{proof}
\begin{rem}
    The analogous result \cite[Fact 15]{Pfaffian_cells} for $\rpfaff$ is not at all easy to prove, and is due to Gabrielov \cite{Frontier_pfaffian}. 
\end{rem}
\subsection{Sharp cellular decomposition, $\Omega^{*}$}
\label{sec:Sharp_cellular_decomposition}
Fix $\cS$ an o-minimal expansion of $\R$. We recall the notion of a \emph{cell}. A
cell $C\subset\R$ is either a point or an open interval (possibly
infinite). A cell $C\subset\R^{\ell+1}$ is either the graph of a
definable continuous function $f:C'\to\R$ where $C'\subset\R^\ell$ is
a cell, or the set
\begin{equation}
\{(x,y)\in C'\times\R|f(x)<y<g(x)\},
\end{equation}
bounded between two graphs of definable continuous
functions $f,g:C'\to\R$ satisfying $f<g$ on $C'$. One can
also take $f\equiv-\infty$ or $g\equiv\infty$ or both in this definition.

We say that a cell $C\subset\R^\ell$ is \emph{compatible} with
$X\subset\R^\ell$ if either $C\subset X$, or $C\cap X=\emptyset$.

\begin{defi}[Cylindrical cellular decomposition]
A cylindrical cellular decomposition of $\R$ is any decomposition of $\R$ into disjoint cells. A decomposition $\{C_j\}$ of $\R^{\ell+1}$ into disjoint cells is a cylindrical cellular decomposition if the collection $\{\pi_{\ell}(C_j)\}$ is a cylindrical cellular decomposition of $\R^{\ell}$.
\end{defi}
For brevity, from now on we will use \say{cellular decomposition} for  \say{cylindrical cellular decomposition}. The following cellular decomposition theorem is the fundamental central result of the classical theory of o-minimal structures. We say that a cellular decomposition is compatible with a set $X$ if every cell from the decomposition is compatible with $X$. 
\begin{thm}[cellular decomposition]\label{thm:cell-decomp}
  Let $X_1,\ldots,X_k\subset\R^\ell$ be definable sets. Then there is
  a cellular decomposition of $\R^\ell$ whose cells are compatible with $X_1,\dots,X_k$.
\end{thm}
\begin{proof}
See \cite[Theorem 2.11]{van_den_dries_book}.
\end{proof}
We are finally in position to define sharp cellular decompositions.
\begin{defi}
\label{defi:sharp_cell_decomposition}
Let $\cS$ be an o-minimal expansion of $\R$ and $\Omega$ be an FD-filtration. We say that $(\cS,\Omega)$ has \emph{sharp cellular decomposition} (or \s CD for short) if for every collection $\{X_{j}\subset\R^{\ell}\}$ of $k$ sets of format $\cF$ and degree $D$, there exists a cellular decomposition of $\R^{\ell}$, compatible with the sets $X_{j}$,  into $\poly_{\cF}(D,k)$ cells of format $\format$ and degree $\degree$. Sometimes we say that $\Omega$ has \s CD if $\cS$ is \emph{fixed} and clear from context, and sometimes we say that $\cS$ has \s CD if $\Omega$ is \emph{fixed} and clear from context.
\end{defi}

\begin{rem}
\label{rem:bad_cellular_decomp}
Following the proof of cellular decomposition as it appears in \cite{van_den_dries_book}, it is possible to recover a cellular decomposition where the formats of the cells are polynomial in the degree (in all three variants of sharpness). Unlike \s CD, the dependence of the format of the cells on the degrees of the initial sets renders this kind of cellular decomposition useless for quantitative applications. The most basic difficulty in constructing compatible cells whose format does not depend on the degree of the initial sets can be illustrated by the following example. Let $X\subset\R^{2}$ be a curve of format $\cF$ and degree $D$, definable in some \so-minimal structure. In general, it seems impossible to \say{select} the middle element of a finite fiber of $\pi_{1}|_{X}$. Indeed, if $K$ is the size of a finite fiber over $x\in\R$, then to pick the middle element $y_{K/2}$ one needs the following formula:
\begin{equation}
    \exists y_{1}\dots\widehat{\exists y_{K/2}}\dots\exists y_{K} \left(y_{1}<\dots<y_{D}\right)\wedge\left((x,y_{1})\in X\wedge\dots\wedge(x,y_{K})\in X\right),
\end{equation}
where the hat signifies that the variable $y_{K/2}$ is not quantified. This formula has format $\geq K$, but generally $K$ grows polynomially in $D$, so the format of this formula depends on $D$. The authors suspect that in general not every \so-minimal structure has \s CD. 

The introduction of $\Omega^{*}$ below is meant to circumvent the problem pointed out above. The idea is that in $\Omega^{*}$ we \emph{force} all connected components $X^{\circ}$ of a set $X\in\Omega_{\cF,D}$ to be in $\Omega_{\format,\degree}$, and that with this condition we can show \s CD.
\end{rem}

Inspired by \cite{Pfaffian_cells}, given a P\so-minimal structure $(\cS,\Omega)$, we will define a new filtration $\Omega^{*}$ such that $\Omega\leq\Omega^{*}$ and $\Omega^{*}$ is equivalent to a sharp filtration with \s CD. We call $\Omega^{*}$ the \emph{star-filtration associated} to $\Omega$.
\begin{defi}[*-format and *-degree]\label{defi:star_fordeg} Let $\cS$ be an o-minimal expansion of $\R$, and let $\Omega$ be an FD-filtration on $\cS$. Let $X\subset\bbR^{\ell}$ be definable. We say that $X\in\Omega^{*}_{\cF,D}$ if there exists a finite collection of sets $X_{\alpha}\subset\R^{\ell_{\alpha}}$ such that $X=\cup_{\alpha}\pi^{\ell_{\alpha}}_{\ell}\left(X^{\circ}_{\alpha}\right)$, where for every $\alpha$ the set $X_{\alpha}$ is in $\Omega_{\cF_{\alpha},D_{\alpha}}$, the set $X^{\circ}_{\alpha}$ is a connected component of $X_{\alpha}$, and $\cF=\max_{\alpha}\{\cF_{\alpha}\}$, $D=\sum_{\alpha}D_{\alpha}$. The star-filtration $\Omega^{*}:=\{\Omega^{*}_{\cF,D}\}_{\cF,D}$ is clearly an FD-filtration, and if $X\in\Omega^{*}_{\cF,D}$, we say that $X$ has *-format $\cF$ and *-degree $D$. 
\end{defi}
The following lemma is crucial, but easy given \Cref{prop:bound_on_H_0}.
\begin{lem}
    \label{lem:compare_with_star}
    Let $(\cS,\Omega)$ be a P\so-minimal structure, and $\Omega^{*}$ the star-filtration associated to $\Omega$. Then $\Omega\leq\Omega^{*}$
\end{lem}
\begin{proof} 
If $X\in\Omega_{\cF,D}$ and $X=\cup^{N}_{i=1}X_{i}$ where $X_{i}$ are the connected components of $X$, then by definition $X$ has *-format $\cF$ and *-degree $N\cdot D$, but we know from \Cref{prop:bound_on_H_0} that $N\le\degree$.
\end{proof}
The following theorem (\Cref{thm:*omega_effective_cellular_decomposition}) is a straightforward generalization of the main result of \cite{Pfaffian_cells}, the proof going through in the general presharp case verbatim, with  \Cref{prop:bound_on_H_0}, \Cref{prop:sharp_definable_choice}, and \Cref{prop:stratification} replacing \cite[Fact 17]{Pfaffian_cells}, \cite[Lemma 18]{Pfaffian_cells}, and \cite[Fact 13]{Pfaffian_cells} respectively. A sketch of the proof of \Cref{thm:*omega_effective_cellular_decomposition} is provided in \Cref{sec:main_proof}.
\begin{thm}
\label{thm:*omega_effective_cellular_decomposition}
Let $(\cS,\Omega)$ be a P\so-minimal structure, and let $\Omega^{*}$ be the star-filtration associated with $\Omega$. Then $\Omega^{*}$ has \s CD. In other words, for every finite collection $X_1,\ldots,X_k\subset\bbR^{\ell}$ of definable sets of *-format $\cF$
  and *-degree $D$, there exists a cellular decomposition of $\R^{\ell}$ compatible with $X_1,\dots,X_k$, where each cell
  has *-format $\format$, *-degree $\degree$, and the number of cells is bounded by $\poly_{\cF}(k,D)$.
\end{thm}
In particular one can show that $\Omega^{*}$ is equivalent to a weakly-sharp FD-filtration on $\cS$ with \s CD, and by \Cref{prop:Wsharp_SCD_isSharp} we conclude that $\Omega^{*}$ is equivalent to a \so-minimal filtration on $\cS$ with \s CD. Thus, we obtain the following theorem, a more precise version of our main result, \Cref{thm:for_show}. A detailed proof of \Cref{prop:main_result} is in \Cref{sec:main_proof}.
\begin{thm}
\label{prop:main_result}
Let $(\cS,\Omega)$ be a P\so-minimal structure. Then $\Omega^{*}$ is equivalent to a filtration making $\cS$ into a \so-minimal structure with \s CD.
\end{thm}
\begin{rem}
In the spirit of \Cref{rem:star_counting}, the above proposition essentially means that one can always assume \s CD when working with a \so-minimal structure. 
\end{rem}
 The following is an easy consequence of \Cref{thm:*omega_effective_cellular_decomposition}.
\begin{prop}
Let $(\cS,\Omega)$ be a \so-minimal structure. Then it has sharp cellular decomposition if and only if $\Omega^{*}\leq\Omega$. In particular, $\Omega^{*}$ is equivalent to $\left(\Omega^{*}\right)^{*}$. 
\end{prop}
\begin{proof}
If $\Omega^{*}\leq\Omega$ then $\Omega,\Omega^{*}$ are equivalent and since $\Omega^{*}$ has \s CD, so does $\Omega$. Assume now that $(\cS,\Omega)$ is a \so-minimal structure with \s CD. Then it follows from \s CD that if $X\in\Omega_{\cF,D}$ and $X^{\circ}\subset X$ is a connected component, then $X^{\circ}\in\Omega_{\format,\degree}$. It now follows immediately from the definition that $\Omega^{*}\leq\Omega$.
\end{proof}
\subsection{Effectivity}
All of the results of this paper are effective in the following sense. 
\begin{defi}
    A \so-minimal (resp. W\so-minimal, P\so-minimal) structure is effective if the polynomial $P_{\cF}(D)$ from axiom (\s1) (resp. (W1),(P1)) is a primitive recursive function of $\cF$. Similarly, a reduction $\Omega\le\Omega'$ is effective if $a(\cF)$ and $P_{\cF}(D)$ from \Cref{defi:reduction_of_filtration} are primitive recursive functions of $\cF$.
\end{defi}
 Assuming effectivity, all the reductions constructed in this paper are effective, and all appearances of $\format,\degree$ (perhaps with dependence on other variables such as $r$) are primitive recursive functions of $\cF$ (and the other variables). 
\subsection{Structure of this paper}
In \Cref{sec:sharp_stratification} we discuss derivatives of functions definable in presharp structures and prove a sharp stratification result. We also review the notion of \emph{sharp derivatives}, first introduced in \cite{Wilkie_conj_Proof}. In \Cref{sec:sharp_definable_choice} we prove a sharp version of definable choice. Note that while a sharp version of definable choice was proved in \cite{Wilkie_conj_Proof}, it was under the assumption of \s CD. Sharp definable choice, along with the results of, \Cref{sec:bound_on_H_0}, \Cref{sec:sharp_stratification} are needed to prove \Cref{thm:*omega_effective_cellular_decomposition}, so we must prove sharp definable choice and the other results from these sections without \s CD. Note that in \cite{Pfaffian_cells} the use of sharp definable choice is circumvented by a similar result \cite{Pfaffian_cells}[Lemma 18] the authors call \say{effective fiber cutting}. In \Cref{sec:bound_on_H_0} we prove \Cref{prop:bound_on_H_0}. In \Cref{sec:main_proof} we provide a sketch of a proof for \Cref{thm:*omega_effective_cellular_decomposition}, and we prove \Cref{prop:Wsharp_SCD_isSharp} and \Cref{prop:main_result}. As an application of \Cref{thm:*omega_effective_cellular_decomposition}, in \Cref{sec:Sharp_Triangulation} we show that a \so-minimal structure with \s CD has sharp triangulation, and deduce bounds on the Betti numbers of a definable set in any presharp structure.
\section{Stratification}
\label{sec:sharp_stratification}
Let $r$ be a positive integer, and fix a presharp structure $\left(\cS,\Omega\right)$. Then we have the following.
\begin{prop}
\label{prop:effectivity_of_derivatives}
Let $f:\bbR^{\ell}\to\bbR^{k}$ be a definable map of format $\cF$ and degree $D$. Then $f$ is $C^{r}$ outside a definable set $V$ of codimension $\geq 1$ of format $\rformat$ and degree $\poly_{\cF}(D,k,r)$.
\end{prop}
 \begin{proof}
 We prove it for $r=1$ and leave the general case for the reader. Fix $1\leq i\leq k$. Then the set $A_i=\{x\in\R^{\ell}|f_{i}$ is differentiable at $x\}$ can be given by the following formula:
 \begin{equation}
     \exists L_{i}\;\forall\epsilon>0\;\exists\delta>0\;\forall y\; \left(|y-x|<\delta\to\left|f_i(y)-f_i(x)-L_i\cdot (y-x)\right|<\epsilon|y-x|\right),
 \end{equation}
 where $L_{i}$ should be understood as a tuple of $\ell$ variables. This formula has $O(\ell)$ variables, and its atoms are all of format $O_{\cF}(1)$ and degree $O(D)$. Moreover, the depth of its parse tree is easily seen to be bounded by $O(\ell)$.
 Thus, by \Cref{prop:fordeg_of_formula_set} $A_i$ has format $O_{\cF}(1)$ and degree $\degree$. Therefore $\cap_{i}A_{i}$ has format $O_{\cF,k}(1)\le O_{\cF}(1)$ and degree $\poly_{\cF}(D,k)$. By o-minimality, $\R^{\ell}\setminus A_{i}$ has codimension $\geq 1$, and thus $V:=\R^{\ell}\backslash\cap_{i}A_{i}$ has codimension $\geq 1$.
\end{proof}
\begin{rem}
    Technically, the degrees of the sets $A_{i}$ in the above proof are bounded by $\poly_{\ell}(D)$, which theoretically could be smaller than $\degree$ (recall axiom (P2)). However, the main idea of this theory is to eventually apply axiom (P1), so that inevitably $\cF$ and $\ell$ will mix. More precisely, if say $\widetilde{A_{i}}$ is the projection of $A_{i}$ to $\R$, then it has format $\cF+\ell-1$ and degree $\poly_{\ell}(D)$, so the number of connected components is still bounded by $\degree$, even when $\cF$ does not appear in the degree of $\widetilde{A_{i}}$. In this sense, the bounds $\poly_{\ell}(D)$ and $\degree$ are of very similar nature, and we will not bother to distinguish between them in the future.
\end{rem}
Given a positive integer $r$ and a set $X\subset\bbR^{\ell}$, it is always possible to stratify $X$ in the following sense. 
\begin{prop}
\label{prop:stratification}
Let $X\subset\bbR^{\ell}$ be a $\mu$-dimensional definable set of format $\cF$, degree $D$, and let $r$ be a positive integer. Then there exists a stratification $X=X_{1}\cup\dots\cup X_{\mu}$ of $X$ where each $X_{i}\in\Omega_{\rformat,\grdegree}$ is a (possibly disconnected) $C^{r}$  smooth embedded submanifold of $\bbR^{\ell}$. 
\end{prop}
\begin{rem}
    Note that \cite[Proposition 7]{Wilkie_conj_Proof} is very similar. The key difference is that in \cite{Wilkie_conj_Proof} the proposition is stated and proved for W\so-minimal structures, while we need it for presharp structures. Moreover, \cite[Proposition 7] {Wilkie_conj_Proof} is proved under the assumption of \s CD (in order to bound the number of connected components of the \say{regular} part of $X$), and thus it results in $\degree$ connected strata, rather than $\mu$ possibly disconnected strata. In the present paper, we first prove the version stated above, and as a corollary, we obtain the same formulation as in \cite[Proposition 7]{Wilkie_conj_Proof}, see \Cref{cor:dy_kvar} below. Finally, by a $C^{r}$-smooth embedded submanifold of $\R^{\ell}$, we mean a set $M\subset\R^{\ell}$, such that for every point $x\in M$ there exists a neighborhood $U$ of $x$ in $\R^{\ell}$ such that $U\cap M$ is the graph of a $C^{r}$ map. 
    \end{rem}
    \begin{proof}[Proof of \Cref{prop:stratification}]
     The proof is by induction on $\mu$. For $\mu=0$ there is nothing to prove. Let $X_{\textnormal{reg}}$ be the set of points $x\in X$ with a neighborhood $U$ such that $X\cap U$ is a $\mu$-dimensional $C^{r}$ embedded submanifold of $\R^{\ell}$. As this can be easily expressed in \say{$\epsilon-\delta$} language, we leave it for the reader to verify that $X_{\textnormal{reg}}$ is definable of format $\rformat$ and degree $\rdegree$. By o-minimality, $X\setminus X_{\textnormal{reg}}$ has dimension $<\mu$, and so for it we can use the induction hypothesis, while $X_{\textnormal{reg}}$ is already a smooth embedded $C^{r}$ submanifold.
    \end{proof}
    \begin{cor}
    \label{cor:dy_kvar}
        Let $X\subset\bbR^{\ell}$ be a $\mu$-dimensional definable set of format $\cF$, degree $D$ and let $r$ be a positive integer. Then there exists a stratification $X=X_{1}\cup\dots\cup X_{s}$ of $X$ where $s=\rdegree$ and each $X_{i}\in\Omega_{\rformat,\grdegree}$ is a connected $C^{r}$ smooth embedded submanifold of $\bbR^{\ell}$. 
    \end{cor}
    \begin{proof}
        This follows immediately from \Cref{prop:stratification} and \Cref{prop:bound_on_H_0}. Indeed, one simply needs to apply \Cref{prop:bound_on_H_0} to every one of the strata of $X$ coming from \Cref{prop:stratification}.
    \end{proof}
%\begin{proof}
%The proof is by induction on dimension. Let $X_{\textnormal{reg}}\subset X$ be the set of points near which $X$ is a $C^{r}$ manifold. By an argument similar to that of \Cref{prop:effectivity_of_derivatives}, $X_{\textnormal{reg}}$ has format $\rformat$ and degree $\grdegree$, and in the case of sharp derivatives, format $\format$ and degree $\grdegree$. By o-minimality $\dim X\backslash X_{\textnormal{reg}}<\dim  X$, so we can apply the induction hypothesis on $X\backslash X_{\textnormal{reg}}$, and we are done.
%\end{proof}
\begin{rem}
In \cite{Forts}, a similar result, more compatible with cellular decomposition, is proved. In particular, we may assume that the $X_{i}$ form a cellular decomposition of $X$.
\end{rem}
Similarly, it follows that the derivatives of $f$, where defined, have format $\rformat$ and degree $\poly_{\cF}(D,r)$. The fact that the format of $f^{(r)}$ depends on $r$ can be very restrictive in applications, and so far it seems generally unavoidable. In structures like $\R_{\alg}$ and $\R_{\rpf}$, however, the format of the derivatives is independent of $r$. We therefore recall the notion of \emph{sharp derivatives}, first introduced in \cite{Wilkie_conj_Proof}.
\begin{defi}
\label{defi:sharp_derivatives}
Let $\cS$ be an o-minimal structure and $\Omega$ be an FD filtration. We say that $(\cS,\Omega)$ has \emph{sharp derivatives} if for every
  $\cF\in\N$ there are
  \begin{align}
    a_\cF&\in\N, &  b_\cF&\in\N[x,y]
  \end{align}
  such that the following holds. For every definable $f:\R^n\to\R$ of format $\cF$ and degree $D$ and for every 
$\alpha\in\Z_{\ge0}^n$, one has
  \begin{equation}
    f^{(\alpha)}\in\Omega_{a_\cF,b_\cF(D,|\alpha|)},
  \end{equation}
 where by $f^{(\alpha)}$ we mean that $f$ is restricted to the locus where it is in $C^{|\alpha|}$.
\end{defi}
\begin{rem}
    In \Cref{prop:stratification} above, if the structure has sharp derivatives, then the strata $X_{i}$ can be taken to have format $\format$ and degree $\poly_{\cF}(D,r)$. Moreover, only $\poly_{\cF}(D,r)$ strata are needed. 
\end{rem}
\section{Sharp definable choice}
\label{sec:sharp_definable_choice}
 Fix a presharp structure. We prove the following form of sharp definable choice. Our proof is completely inspired by the construction in \cite{van_den_dries_book}.
\begin{prop}
 \label{prop:sharp_definable_choice}
 Let $\{X_{\lambda}\subset\bbR^{\ell}\}_{\lambda\in\Lambda}$ be a definable family whose elements are nonempty, such that the format of the total space $X_{\Lambda}:=\{(\lambda,x)|\;x\in X_{\lambda},\;\lambda\in\Lambda\}$ is $\cF$ and its degree is $D$. Then there exists a definable map $g:\Lambda\to\bbR^{\ell}$ of format $\format$ and degree $\degree$ such that $g(\lambda)\in X_{\lambda}$ for all $\lambda\in\Lambda$.
 \end{prop}
 \begin{proof}
 We prove this by induction on $\ell$. First we show the induction step. We use the induction hypothesis on the family $\pi_{\ell-1}(X_{\lambda})\subset \bbR^{\ell-1}$ to obtain a map $g_1:\Lambda\to\R^{\ell-1}$ with $g_1(\lambda)\in \pi_{\ell-1}\left(X_{\lambda}\right)$ for every $\lambda\in\Lambda$. Since $g_1$ has format $\format$ and degree $\degree$, the total space of the family $\{x\in\R: (g_{1}(\lambda),x)\in X_{\Lambda}\}_{\lambda\in \Lambda}$ has format $\format$ and degree $\degree$. Finally, we use the induction hypothesis on this family, producing a map $g_2:\Lambda\to\R$ of format $\format$ and degree $\degree$. Then the map $(g_1,g_2)$ satisfies the requirements of the proposition. \\ 

Finally, we show the statement for $\ell=1$. For any definable $X\subset\bbR$, define $a(X):=\inf X$ and $b(X):=\sup\{x:(a(X),x)\subset X\}$. We decompose $\Lambda$ into the following sets, and define $g$ on them:
\begin{equation*}
\begin{aligned}
    A_{1}=&\{\lambda:a(X_{\lambda})=-\infty,\;b(X_{\lambda})=\infty\},\;g(\lambda)=0, \\
    A_{2}=&\{\lambda:a(X_{\lambda})=-\infty,\; b(X_{\lambda})\in\bbR\},\;g(\lambda)=b(X_{\lambda})-1, \\
    A_{3}=&\{\lambda:a(X_{\lambda})\in\bbR,\;b(X_{\lambda})=\infty\},\;g(\lambda)=a(X_{\lambda})+1, \\
    A_{4}=&\{\lambda:a(X_{\lambda}),b(X_{\lambda})\in\bbR\},\;g(\lambda)=\frac{a(X_{\lambda})+b(X_{\lambda})}{2}.
\end{aligned}
\end{equation*}
It remains to check that $A_{1}, A_{2}, A_{3}, A_{4}$, as well as the functions (of the variable $\lambda$) $a(X_{\lambda}),b(X_{\lambda})$ are of format $\format$ and degree $\degree$. We will only check this for $A_{2}$ and $a(X_{\lambda})$ when restricted to $A_{3}$, and leave the rest for the reader. $A_{2}$ can be described by the formula \begin{equation}
    \left(\forall M\;\exists x\in X_{\lambda}\;x<M\right)\wedge\left(\exists N\;\forall x\in X_{\lambda}\; x<N\right),
\end{equation}
thus due to \Cref{prop:fordeg_of_formula_set}, the format and degree of $A_{2}$ are bounded by $\format, \degree$ respectively.  We proceed to bound the format and degree of $a(X_{\lambda})$ on $A_{3}$. The graph of $a(X_{\lambda})$ on $A_{3}$ is given by the set \begin{equation}
    \{(\lambda,a)\in A_{3}\times\R:\left(\forall x\in X_{\lambda}\;a\leq x\right)\wedge\left(\forall\epsilon>0\exists z\in X_{\lambda}\;a+\epsilon>z\right)\}
\end{equation}
Thus, just as before, the format and degree of $a(X_{\lambda})$ are bounded by $\format, \degree$ respectively. 
 \end{proof}
 \section{Bound on connected components of definable sets in P\so-minimal structures}
\label{sec:bound_on_H_0}
Fix a presharp structure $\left(\cS,\Omega\right)$. The goal of this section is to prove \Cref{prop:bound_on_H_0}. 
\begin{defi}
    For a linear functional $f\in\left(\bbR^{\ell}\right)^{*}$ and a set $X\subset\bbR^{\ell}$, we denote by $X_{f}$ the set of local maxima of $f|_{X}$.
\end{defi}
Note that, in the definition above, if $X$ is definable then so is $X_{f}$. We will need the following lemma, which is a standard exercise in o-minimality.
\begin{lem}
\label{lem:dim_of_local_maxima}
 Let $X\subset\R^{\ell}$ be definable and suppose that $\dim X>0$. Then there exists a functional $f\in\left(\R^{\ell}\right)^{*}$ such that $\dim X_{f}<\dim X$.
\end{lem}
\begin{proof}[Proof of \Cref{prop:bound_on_H_0}]
Suppose first that $X\subset D(R)$, where $D(R)$ is an open disk of radius $R$ around the origin. Let $X=X_{1}\cup\dots\cup X_{N}$ be the decomposition of $X$ into its connected components. The proof is by induction on $\dim X$. The idea is to use \Cref{lem:dim_of_local_maxima} and the induction hypothesis, but the problem is that the $X_{i}$ may have intersecting closures. \\ 

If $\dim X=0$, then there exists a linear functional $f:\bbR^{\ell}\to\bbR$ such that $f|_{X}$ is injective. Thus $X$ and $f(X)$ are of the same (finite) cardinality, but $f(X)$ has format $\format$ and degree $\degree$, so according to (P1), $f(X)$ has at most $\degree$ connected components. \\ 

Now suppose $\dim X>0$. Let $\epsilon>0$, consider $X_{\epsilon}:=\{x\in X|\;d(x,\partial X)<\epsilon\}$ and $Y:=X\backslash X_{\epsilon}$. Clearly $Y$ has format $\format$ and degree $\degree$. We claim that for $\epsilon$ small enough, $Y$ is a union of $N$ sets with disjoint closures. Define $Y_{i}=X_{i}\backslash X_{\epsilon}$, and to ensure that the $Y_{i}$ are nonempty, choose $\epsilon<\underset{i}{\min}\underset{x\in X_{i}}{\sup}{d(x,\partial X)}$. Note that this can be done. Suppose to the contrary that for some $i$ one has $\underset{x\in X_{i}}{\sup}d(x,\partial X)=0$, then $X_{i}$ must belong to  $\overline{\partial X}$. Since a connected component cannot intersect the closure of other connected components, it follows that $\partial X=\cup_{j=1}^{n}\partial X_{j}$. Moreover, it follows that $X_{i}$ must be contained in  $\overline{\partial X_{i}}$, which is impossible since frontier reduces dimension. \\ 

Let us prove that the $Y_{i}$ have disjoint closures. Say $x\in\overline{Y_{i}}\cap\overline{Y_{j}}$, and so $x\in\overline{X}$. If $x\notin X$, then $x\in\partial X$, but $d(\overline{Y_{i}},\partial X)\geq\epsilon$, contradicting $x\in\overline{Y_{i}}$. So $x\in X$, say $x\in X_k$, but since $x\in\overline{X_{i}}$ we conclude that ${X_{i}}\cup{X_{k}}$ is connected. This forces $k=i$, and by repeating the argument that $k=j$. Of course this shows that $Y$ has at least $N$ components. \\ \\ Recall that the zeroth homology counts connected components, in the sense that $H_0(X;\bbR)$ is a vector space with dimension equal to the number of (path connected) components of $X$.
If $\dim Y=\dim\overline{Y}=0$, then we are done because $Y$ has at least $N$ components on the one hand, and on the other hand it has at most $\degree$ components due to the induction hypothesis. Suppose that $\dim{\overline{Y}}>0$, by \Cref{lem:dim_of_local_maxima} there exists a functional $f$ such that $\dim\left(\overline{Y}\right)_{f}<\dim Y$. Certainly, $\left(\overline{Y}\right)_{f}$ has format $\format$ and degree $\degree$. Moreover, since $Y$ is bounded, $\left(\overline{Y}\right)_{f}$ meets every $\overline{Y}_{i}$, and since the $\overline{Y}_{i}$ are disjoint we conclude $N\leq\dim H_{0}\left(\left(\overline{Y}\right)_{f};\R\right)$. As  $\dim\left(\overline{Y}\right)_{f}<\dim Y\leq \dim X$, by induction $N\leq\dim H_{0}(\left(\overline{Y}\right)_{f};\R)\leq\degree$. \\ 

 To end the proof, note that while we assume that $X\subset D(R)$, our bound on $\dim H_0(X;\bbR)$ does not depend on $R$. Assume now that $X$ is not bounded.  Since homology commutes with direct limit, we see that $H_0(X;\bbR)=\underset{\longleftarrow}{\lim}H_0(X\cap D(R);\bbR)$, and so $\dim H_0(X;\bbR)\leq\limsup\dim H_0(X\cap D(R);\bbR)\le\degree$. (The direct limit here is taken with respect to the direct system of inclusions $X\cap D(R)\subset X\cap D(R')$ when $R'>R$.)
\end{proof}
 \section{Proof of the main results}
\label{sec:main_proof}
\subsection{Sketch of the proof of \Cref{thm:*omega_effective_cellular_decomposition}}
\label{subsec:main_proof_1}
We will explain the main ideas and the main steps of the proof, while trying to avoid as many technical details as possible. The interested reader should consult \cite{Pfaffian_cells}[Theorem 1, Proposition 19]. We may replace $\R$ by $I$, using the definable homeomorphism $\frac{x-1/2}{x-x^2}:I\to\R$. For the remainder of \Cref{subsec:main_proof_1}, fix a presharp structure. The following proposition is the key ingredient in the proof.  
\begin{prop}
\label{prop:weird_prop}
 Let $\{X_{\alpha}\}$ be a collection of $N$ definable sets of format $\cF$ and degree $D$ in $I^{\ell}$. Let $n\leq\ell$ be a positive integer, then there exists a cellular decomposition of $I^{n}$ of size $\poly_{\cF}(D,N)$, compatible with the collection $\{\pi_{n}(X_{\alpha})\}$, and whose cells have *-format $\format$ and *-degree $\degree$.   
\end{prop}
Let us show how \Cref{thm:*omega_effective_cellular_decomposition} follows from \Cref{prop:weird_prop}. Given $X_{1},\dots,X_{k}\subset I^{n}$ in $\Omega^{*}_{\cF,D}$, according to \Cref{defi:star_fordeg}, each $X_{i}$ is a union of projections of connected components of definable sets $X_{\alpha,i}\in\Omega_{\cF_{\alpha,i},D_{\alpha,i}}$, and $\cF=\underset{\alpha}{\max}\cF_{\alpha,i},\;D=\underset{\alpha}\Sigma D_{\alpha,i}$ for every $i$. We may suppose without loss of generality that the $X_{\alpha,i}$ are contained in the same ambient space $I^{\ell}$ (by replacing them with $X_{\alpha,i}\times I^{t}$ for a suitable $t$). Indeed, according to (P2) this will increase the format of every $X_{\alpha,i}$ by at most $\cF$. Now we can use \Cref{prop:weird_prop} on all the sets $X_{\alpha,i}$, producing a cellular decomposition compatible with $X_{\alpha,i}$ for every $\alpha$ and $i$. Then it must also be compatible with the connected components of $X_{\alpha,i}$, and the projections of the cells to $I^{n}$ will be compatible with the sets $X_{i}$.

We next turn to prove \Cref{prop:weird_prop}. The proof is by lexicographic induction on $(n,k)$ where $k:=\underset{\alpha}{\max}\dim\left(\pi_{n-1}(X_{\alpha})\right)$. For a collection $\Pi$ of subsets of $I^{\ell}$, let $\bigcup\Pi$ be the union of the sets in $\Pi$, and for a positive integer $n\leq\ell$ denote $\pi_{n}(\Pi):=\{\pi_n(X)|X\in\Pi\}$. For a positive integer $t$, we denote by $\Pi_{<t}$ and $\Pi_{t}$ to be the collection of sets in $\Pi$ whose dimension is smaller than $t$ and whose dimension is equal to $t$, respectively, and moreover let $\Pi_{\leq t}:=\Pi_{<t}\cup\Pi_{t}$ denote their union. \\ 
\textbf{Step 1}: By \Cref{prop:sharp_definable_choice}, we may assume that for every $\alpha$ the projection $\pi_{n}|_{X_{\alpha}}$ has finite fibers . We stratify the sets $X_{\alpha}$ as in \Cref{prop:stratification}, so that we may assume that the collection $\Pi:=\{X_{\alpha}\}$ consists of $C^{1}$-smooth pure dimensional manifolds, and for a technical reason we also use \Cref{lem:frontier} in order to assume that $\Pi$ is closed under taking frontiers. Since for each $\alpha$ the fibers of $\pi_{n}|_{X\alpha}$ are zero-dimensional, we may also assume that for every $\alpha$ the map $\pi_{n}|_{X_{\alpha}}$ has constant rank $\dim\pi_n(X_{\alpha})$, and that the map $\pi_{n-1}|_{X_{\alpha}}$ has constant rank $k$. The first step is to remove the singularities of $\cup{\pi_{n-1}(X_{\alpha})}$ by removing sets of dimension $<k$. More specifically, we will define a collection $\cB$ of closed definable sets of bounded *-format and *-degree such that $\dim\bigcup\pi_{n-1}(\cB)<k$. The idea is that $\Pi$ is well behaved above $\bigcup\pi_{n-1}(\Pi)\setminus\bigcup\pi_{n-1}(\cB)$, so that a cellular decomposition can be explicitly constructed, and over $\bigcup\pi_{n-1}(\cB)$ we will simply construct cells by induction. We begin by putting the closures of the sets in $\Pi_{<k}$ into $\cB$.

By sharp definable choice, i.e by \Cref{prop:sharp_definable_choice} we may assume that $\bigcup\pi_{n-1}(\Pi_{k+1})\subset\bigcup\pi_{n-1}(\Pi_{<k}\cup\Pi_{k})$, and so for the purpose of this step we may ignore $\Pi_{k+1}$. Let us analyze $Y:=\bigcup\pi_{n-1}(\Pi_k)$. As $\pi_{n-1}|_{X_{\alpha}}$ has constant rank $k$, $Y$ can be thought as an immersed manifold with self-intersections, so if a point $p\in Y$ is not smooth then one of the following must hold.
\begin{enumerate}
    \item There are two sets $X_{\alpha},X_{\beta}\in\Pi_k$ and two points $p_{\alpha}\in X_{\alpha},\;p_{\beta}\in X_{\beta}$ such that $\pi_{n-1}(p_{\alpha})=\pi_{n-1}(p_{\beta})=p$, but  the projections of the germs $(X_{\alpha},p_{\alpha}),\;(X_{\beta},p_{\beta})$ to $I^{n-1}$ are different.
    \item The point $p$ is in $\bigcup\pi_{n-1}\left(\{\partial X|X\in\Pi_k\}\right)$.
\end{enumerate}
The second case is handled by induction - since frontier reduces dimension, we may add the frontiers of the sets in $\Pi_{\le k}$ to $\cB$. Let us now explain how to deal with the first case. Consider the following set,
\begin{equation}
    X_{\alpha,\beta}:=\{(x,y)\in X_{\alpha}\times X_{\beta}|x_1=y_1,\dots,x_{n-1}=y_{n-1}\}\subset I^{\ell}\times I^{\ell},
\end{equation}
which in particular contains the point $(p_{\alpha},p_{\beta})$. Let us stratify $X_{\alpha,\beta}$. We claim that $(p_{\alpha},p_{\beta})$ must lie in a stratum of dimension $<k$. Indeed, first notice that $X_{\alpha,\beta}$ has discrete fibers over $\pi_{n-1}(X_{\alpha})$ (under projection on the coordinates $x_1,\dots,x_{n-1})$, and therefore its strata are of dimension at most $k$. Moreover, $(p_{\alpha},p_{\beta})$ cannot lie in a $k$-dimensional stratum of $Y$, since then the germ of $Y$ at $(p_{\alpha},p_{\beta})$ is diffeomorphically mapped to both the germ of $\pi_{n-1}(X_{\alpha})$ and the germ of $\pi_{n-1}(X_{\beta})$ at $p$, which were assumed to be different. A contradiction. 

We now add the strata of $X_{\alpha,\beta}$ of dimension $<k$ into $\cB$. Note crucially that to remove the singularities of $\cup{\pi_{n-1}(X_{\alpha})}$ we only had to intersect two sets per singular point (of course, while the germ at a singular point $p$ can be a union of more than two projections of germs of sets $X_{\alpha}\in\Pi_{k}$, it only takes two of them to remove $p$). Moreover, note that there are polynomially many such pairs, so to get the required bounds on the number and complexity of the strata it is enough to assume that the structure was presharp.

We finish this step by adding some more sets to $\cB$ for future use. For any $X_{\alpha},X_{\beta}\in\Pi_{k}$ consider
\begin{equation}
    Z_{\alpha,\beta}:=\{(x,y)\in X_{\alpha,\beta}| x_n=y_n\}.
\end{equation}
We stratify $Z_{\alpha,\beta}$ and add the strata of dimension $<k$, as well as the frontiers of the $k$ dimensional strata to $\cB$. \\
\textbf{Step 2}:
We go on to construct the cylindrical decomposition. We will use the induction hypothesis to construct a cylindrical decomposition of $I^{\ell-1}$ compatible with $\pi_{n-1}(\Pi)$ and $\pi_{n-1}(\cB)$. Since $\pi_{n-1}(\cB)$ has dimension $<k$, if $\cC$ is a cell contained in $\pi_{n-1}(\cB)$, we construct the required cells over $\cC$ by induction. Now let $\cC$ be a cell contained in $\bigcup\pi_{n-1}(\Pi)\backslash\pi_{n-1}(\cB)$. We need to construct cells over $\cC$ compatible with $\pi_{n}(\Pi)$, and it is not hard to see that it suffices for the cells to be compatible with $\pi_{n}(\Pi_{k})$ in order to be compatible with $\pi_{n}(\Pi)$.

We claim that the set $\cup_{\alpha}\pi_{n}(X_{\alpha})\cap\left(\cC\times I\right)$ is a union of graph cells of the form $\{(x,y):x\in\cC, y=s(x)\}$. Indeed, let $s_\alpha,s_\beta:\cC\to I$ be sections of $\left(\pi_{n}(X_{\alpha})\right)\cap\left(\cC\times I\right),\left(\pi_{n}(X_{\beta})\right)\cap\left(\cC\times I\right)$, respectively. If neither of the three conditions $s_{\alpha}<s_{\beta},s_{\alpha}=s_{\beta},s_{\alpha}>s_{\beta}$ holds globally over $\cC$, then there exists a point $x\in\cC$ such that $s_{\alpha}(x)=s_{\beta}(x)$, but $s_{\alpha},s_{\beta}$ are not identically equal in a neighborhood of $x$ in $\cC$. Then there is a point $(x,s_{\alpha}(x),x_{1},x,s_{\beta}(x),x_{2})$ in $X_{\alpha,\beta}$ that lies on a stratum of $Z_{\alpha,\beta}$ of dimension $< k$, or it lies on the frontier of a $k$-dimensional stratum of $Z_{\alpha,\beta}$, and in either case, this contradicts $\cC\subset\bigcup\pi_{n-1}(\Pi)\backslash\pi_{n-1}(\cB)$. Again, note crucially that we have achieved this while working with at most two sets $X_{\alpha},X_{\beta}$ at any time, so the construction holds in the presharp case. 

Now, let $s_1<\dots<s_{q}$ be all the sections of all the sets $\pi_{n}(X_{\alpha})$ over $\cC$. Then the required cells are just given by 
\begin{multline}
    \{(x,y)\in\cC\times I|0<y<s_{1}(x)\},\;\{(x,y)\in\cC\times I|y=s_1(x)\},\;\dots\\
    \dots,\{(x,y)\in\cC\times I|s_{q}(x)<y<1\}.
\end{multline}
Note crucially that all these sections come from connected components of the sets $\pi_{n}\left(X_{\alpha}\right)\cap\left(\cC\times I\right)$, and so by \Cref{prop:bound_on_H_0} we have $q\leq N\cdot\degree$. Finally, with some technicalities, one can bound the *-format and *-degree of these cells. Note once more that each cell is defined by at most two sections, so the construction provides the required bounds under the assumption of P\so-minimality.
\subsection{Structure Trees}
\label{subsection:Structure_Trees}
In this section, we define the notion of a \emph{structure tree}. It is similar to a parse-tree for a first order formula, but the operations allowed are in correspondence to the operations in the axioms of o-minimality, rather than to the standard operations for first-order formulae. Let us fix some notation. A \emph{rooted tree} is a pair $(T,r)$ where $T$ is a tree and $r$ is a vertex of $T$. If $v$ is a vertex of $T$, we consider the neighbors $w$ of $v$ such that $d(w,r)=d(v,r)+1$, and refer to such neighbors as \emph{children}. \\ 
We will now define structure trees, and for future use, we will also define \emph{slanted} structure trees. Fix an o-minimal structure $\cS$.
\begin{defi}[Structure trees]
\label{defi:structure_trees}
A structure tree in $\cS$ is a finite rooted tree $(T,r)$ where for every vertex $v$ of $T$ there is an associated definable set $T_{v}$, such that for every vertex $v$ with children $v_{1},\dots,v_{k}$, the following holds:
\begin{enumerate}
    \item The sets $T_{v_{j}}$ for $j=1,\dots,k$ have the same ambient dimension, i.e there exists a natural number $\ell$ such that $T_{v_{j}}\subset\R^{\ell}$ for every $j$.
    \item If $k>1$, then $T_{v}$ is either the union of, or the intersection of, the sets $T_{v_{1}},\dots,T_{v_{k}}$.
    \item If $k=1$, then $T_{v}$ is one of the following: $\pi_{\ell-1}(T_{v_{1}}),\left(T_{v_{1}}\right)^{c},T_{v_{1}}\times\R$.
\end{enumerate}
A slanted structure tree is defined similarly to a structure tree, the only difference is that in item (3) above, the set $T_{v}$ can also be $\R\times T_{v_{1}}$. We sometimes do not explicitly mention the o-minimal structure in which the structure tree is defined, if it is clear from context.
\end{defi}
Given an FD filtration on $\cS$, we can extend it to filter structure trees in several natural ways. For the purposes of this paper, we need the following definition. 
\begin{defi}[degree and $\Omega$-format of structure trees and slanted structure trees]
Let $\Omega$ be any FD-filtration on $\cS$. We define the $\Omega$-format of structure trees by induction. If $T$ has a single vertex $r$ and the associated set $T_{r}$ is in $\Omega_{\cF,D}$, then the $\Omega$-format of $T$ is defined to be $\cF$.

Let $(T,r)$ be a structure tree, let $v_{1},\dots,v_{k}$ be the children of $r$, and denote by $T_{1},\dots,T_{k}$ the subtrees defined by them. Suppose that $\left(T_{i},v_{i}\right)$ has $\Omega$-format $\cF_{i}$, then: 
\begin{enumerate}
    \item If $k=1$ and $T_{r}=T_{1}\times\R$ ($T_{r}=T_{1}\times\R$ or $T_{r}=\R\times T_{1}$ in the slanted case), then the $\Omega$-format of $T$ is $\cF_{1}+1$.
    \item In any other case, the $\Omega$-format of $T$ is $\max\{\cF_{i}\}$.
\end{enumerate} 
If $A_{j}\in\Omega_{\cF_{j},D_{j}}$ are the sets associated to the leaves of $T$, then the degree of $T$ is defined to be $\sum D_{j}$.
\end{defi} 
\begin{rem}
\label{rem:leaves_are_good}
     Suppose that $(\cS,\Omega)$ is a \so-minimal structure, and let $(T,r)$ be a structure tree in $\cS$ with degree $D$ and $\Omega$-format $\cF$. Then $T_{r}$ is in $\Omega_{\cF,D}$, and in fact, this is the motivation of this definition. Moreover, the following holds for any FD-filtration $\Omega'$ on $\cS$. Denote as before the degree of $(T,r)$ by $D$ (but now with respect to the arbitrary filtration $\Omega'$). If $A_{1},\dots,A_{N}$ are the sets associated to the leaves of $T$, then $N\leq D$ and moreover, for every $1\leq i\leq N$ we have $A_{i}\in\Omega'_{\cF,D}$. 
     %Finally, it is useful to note that the $\Omega$-format of $T$ bounds the depth of $T$ from above. 
     We leave the verification of these comments to the reader. 
\end{rem}
The following proposition is the key ingredient in the proof of \Cref{prop:Wsharp_SCD_isSharp}.
\begin{prop}
 \label{prop:ST_defines_good_set}
 Let $(\cS,\Omega)$ be a W\so-minimal structure with \s CD, and let $(T,r)$ be a structure tree in $\cS$ of $\Omega$-format $\cF$ and degree $D$. Then $T_{r}\in\Omega_{O_{\cF}(1),\degree}$. 
\end{prop}
\begin{proof}
Let $A_{i}\subset\R^{\ell_i}$ be the sets associated to the leaves of $T$. Denote $m:=\max\{\ell_i\}$, and for a set $X\subset\R^{m}$ and an integer $\ell$ we denote $P_{\ell}(X):=\pi_{\min\{\ell,m\}}(X)\times\R^{\max\{\ell-m,0\}}$. Let $\cC_1,\dots,\cC_{N}$ be a cellular decomposition of $\R^{m}$ compatible with the sets $A_{i}\times\R^{m-\ell_{i}}$. We claim that for every vertex $v$ of $T$, if the associated set $T_{v}$ is a subset of $\R^{\ell}$, then $P_{\ell}(\cC_1),\dots,P_{\ell}(\cC_{N})$ is a cellular decomposition of $\R^{\ell}$ compatible with $T_{v}$. 

We prove this by descending induction on the distance from $v$ to $r$. If $v$ is a leaf then the claim is clear by definition. Now let $v$ be any vertex, and let $v_1,\dots,v_k$ be its children. Then by the definition of structure trees, one of the following holds.
\begin{enumerate}
    \item If $k>1$, then $T_{v}$ is either the union of, or the intersection of, the sets $T_{v_{1}},\dots,T_{v_{k}}$, but by the inductive hypothesis $P_{\ell}(\cC_{1}),\dots,P_{\ell}(\cC_{N})$ are compatible with $T_{v_{1}},\dots,T_{v_{k}}$, and thus they are compatible with $T_{v}$. 
    \item If $k=1$, then a straightforward check shows the same conclusion.
\end{enumerate}
Since we assumed $(\cS,\Omega)$ has \s CD, and due to the observations made in \Cref{rem:leaves_are_good}, the cells $\cC_{i}$ can be chosen to have format $\format$, degree $\degree$ and their number $N$ can be chosen to be $\degree$. Assume that $T_{r}\subset\R^{n}$, then $\cF\geq n$ (this can be verified by induction on $T$), and the cells $P_{n}(\cC_{i})$ have format $O_{\cF,n}(1)\le O_{\cF}(1)$ and degree $\degree$. Since $T_{r}$ is a union of some of these cells, and crucially since $(\cS,\Omega)$ is W\so-minimal (so that unions do not increase format), $T_{r}$ has format $O_{\cF}(1)$ and degree $\degree$ as well. 
\end{proof}
We end this subsection by noting that slanted structure trees can be used for a constructive definition of sharply generated FD-filtrations, recall \Cref{lem:generated_filtrations}.
\begin{defi}
\label{defi:Omega_ST}
    Let $\cS$ be an o-minimal structure, and $\Omega$ be any FD-filtration on $\cS$ that satisfies axioms (P2) and (P6). We define a new filtration $\Omega^{ST}$ in the following way. A definable set $X$ is in $\Omega^{ST}_{\cF,D}$ if there exists a \emph{slanted} structure tree $\left(T,r\right)$ of $\Omega$-format $\cF$ and degree $D$ such that $T_{r}=X$. 
\end{defi}
Clearly, $\Omega^{ST}$ is an FD-filtration satisfying all the sharp axioms except possibly for (\s 1).
\begin{lem}
    \label{lem:ST_generates_filtrations}
In the notation of \Cref{defi:Omega_ST}, the filtration $\Omega^{ST}$ coincides with the FD-filtration $\Omega'$ sharply generated by $\Omega$.
\end{lem}
\begin{proof}
    The proof is straightforward. One the one hand, $\Omega\subset\Omega^{ST}$ because given a definable set $X$, one can consider a single vertex tree with associated set $X$. Thus, by the minimality of $\Omega'$ we must have $\Omega'\subset\Omega^{ST}$. On the other hand, if $(T,r)$ is a slanted structure tree of $\Omega$-format $\cF$ and degree $D$, then $T_{r}$ is in $\Omega'_{\cF,D}$. This follows immediately from the definition of structure trees, since $\Omega\subset\Omega'$ and $\Omega'$ is a sharp filtration (except possibly for (\s 1)). The latter statement however is exactly equivalent to $\Omega^{ST}\subset\Omega'$.
\end{proof}
\begin{rem}
    In the statement of \Cref{lem:generated_filtrations}, only (P2) is assumed while (P6) is not. Thus, the procedure described above is not sufficient to construct sharply generated filtrations out of general datum. In order to mimic the construction above to work in this greater generality, one would have to, for every semialgebraic set $X\subset\R^{\ell}$ of degree $D$, add a special one-vertex structure tree $(T,r)$ with $T_{r}=X$, and declare that this tree has format $\ell$ and degree $D$.
\end{rem}
\subsection{Proof of \Cref{prop:Wsharp_SCD_isSharp}}
Let $(\cS,\Omega)$ be W\so-minimal with \s CD. We define an FD-filtration $\Omega'$ in the following way. The set $X\subset\R^{n}$ is in $\Omega'_{\cF,D}$ if there exists a structure tree $(T,r)$ of $\Omega$-format $\cF$ and degree $D$ such that $T_{r}=X$ (note the difference with \Cref{lem:ST_generates_filtrations} above, where we also used slanted structure trees). It is clear that $(\cS,\Omega')$ satisfies the axioms of \so-minimal structures, except possibly for (\s 1), and the part of axiom (\s3) having to do with multiplication by $\R$ from the left, i.e. that if $A\in\Omega'_{\cF,D}$ then $\R\times A\in\Omega'_{\cF+1,D}$. The latter follows from the simple observation that if $(T,r)$ is any structure tree, one can replace all the associated sets $T_{v}$ by $\R\times T_{v}$ and obtain a new structure tree of $\Omega$-format greater by $1$ than the $\Omega$-format of $T$.

It is also clear that $\Omega\subset\Omega'$, since given a definable set $X$, one can consider the tree $T$ which is a single vertex with associated set $X$. In particular, $\Omega\leq\Omega'$, and moreover, by \Cref{prop:ST_defines_good_set} we have $\Omega'\le\Omega$, thus we simultaneously obtain the following. 
\begin{enumerate}
    \item The FD-filtration $\Omega'$ satisfies (\s1), or in other words, $(\cS,\Omega')$ is \so-minimal.
    \item The filtrations $\Omega,\Omega'$ are equivalent, therefore $(\cS,\Omega')$ has \s CD. 
\end{enumerate}
\qedsymbol\\
Before turning to the proof of \Cref{prop:main_result}, we need the following lemma on the normalization of format and degree.
\begin{lem}
\label{lem:format_normalized}
Let $\cS$ be an o-minimal expansion of $\R$, and let $\Omega$ be an FD-filtration such that $(\cS,\Omega)$ satisfies the following axioms. \\
If $A\in \Omega_{\cF,D}$ then:
\begin{enumerate}
    \item[(*1)] If $A\subset\bbR$, it has at most $\degree$ connected components.
    \item[(*2)] If $A\subset\bbR^{\ell}$ then $\cF\geq\ell$.
    \item[(*3)] If $A\subset\bbR^{\ell}$ then $\pi_{\ell-1}(A), A\times\bbR,\bbR\times A$ are in $\Omega_{\cF+1,D}$, while $A^{c}\in\Omega_{\format,\degree}$. 
\end{enumerate}
If $A_1,\dots,A_{k}\subset\bbR^{\ell}$ with $A_{i}\in\Omega_{\cF_i,D_i}$ and $\cF:=\max_{i}\{\cF_{i}\},\;D:=\sum_{i}D_i$, then:
\begin{enumerate}
    \item[(*4)]  $\cup_{i}A_{i}\in\Omega_{\cF,D}$,    
    \item[(*5)]  $\cap_{i}A_{i}\in\Omega_{\format,\degree}$.
\end{enumerate}
If $P\in\bbR[x_1,\dots,x_{\ell}]$, then:
\begin{enumerate}
    \item[(*6)] $\{P=0\}\in\Omega_{O_{\ell}(1),\poly_{\ell}(\deg P)}$.
\end{enumerate}
Then $\Omega$ is equivalent to a filtration $\Omega'$ such that $(\cS,\Omega')$ is W\so-minimal. 
\end{lem}
\begin{rem}
    As will be seen from the proof, it is more generally true that in any axiom schema of \so-minimality, axioms that increase format and degree by some function as in the (*)-axioms above, can be reduced via equivalence to axioms that increase the format by $1$. 
\end{rem}
\begin{proof}[Proof of \Cref{lem:format_normalized}]
Denote the functions $\format$ (and $O_{\ell}(1)$ in (*6)) appearing in axioms (*3), (*5) and (*6) by $C_{3},C_{5},C_{6}:\N\to\N$ respectively, and denote the polynomials $\degree$ (and $\poly_{\ell}(\deg P)$ in (*6)) appearing in axioms (*3), (*5) and (*6) by $T_{\cF},Q_{\cF},H_{\ell}$ respectively. \\ 
Define a new FD-filtration $\Omega'$ by $\Omega'_{\cF,D}:=\Omega_{C(\cF),P_{\cF}(D)}\cap\left(\bigcup_{i=1}^{\cF}P(\R^{i})\right)$ where $C:\N\to\N$ and $\{P_{\cF}\}_{\cF\in\N}\subset\N[D]$ will be determined later. Namely, $A$ is in $\Omega'_{\cF,D}$ if $A$ is in $\Omega_{C(\cF),P_{\cF}(D)}$ and $A$ is a subset of one of $\R,\dots,\R^{\cF}$. We will check what are the conditions on $C,P_{\cF}$ that are needed to assure that $(\cS,\Omega')$ is W\so-minimal. It will then be clear that such $C,P_{\cF}$ can be constructed by induction on $\cF$ so that they will satisfy these properties. Note that by definition $\Omega'\leq\Omega$. If we assume in addition that $C(\cF)\geq\cF$ and $P_{\cF}(D)\geq D$, then we obtain $\Omega\subset\Omega'$, so under this assumption we indeed have $\Omega\simeq\Omega'$. \\ 

To guarantee that $\Omega'$ is a filtration, the function $C$ has to be monotone increasing, and per $\cF$ the polynomial $P_{\cF}(D)$ has to be a monotone increasing function of $D$. Let us now carefully check axioms (W1)-(W6). \begin{itemize}
    \item[(W1)] Let $A\subset\R$ be in $\Omega'_{\cF,D}$, so by definition $A$ is in $\Omega_{C(\cF),P_{\cF}(D)}$, and hence has at most $\poly_{C(\cF)}(P_{\cF}(D))=\degree$ connected components. 
    \item[(W2)] Let $A\subset\R^{\ell}$ be in $\Omega'_{\cF,D}$, so by definition $A$ is a subset of one of $\R,\dots,\R^{\cF}$, and in particular $\cF\geq\ell$. 
    \item[(W3)] Let $A\subset\R^{\ell}$ be in $\Omega'_{\cF,D}$, so by definition $A$ is in $\Omega_{C(\cF),P_{\cF}(D)}$. Hence, $\R\times A, A\times\R$ and $\pi_{\ell-1}(A)$ are in $\Omega_{C(\cF)+1,P_{\cF}(D)}$, and if we add the assumption that $C(\cF+1)\geq C(\cF)+1$, we would conclude that $\R\times A, A\times\R,\pi_{\ell-1}(A)$ are in $\Omega'_{\cF+1,D}$. As for $A^{c}$, according to (*3) we have $A^{c}\in\Omega_{C_{3}(C(\cF)),T_{C(\cF)}(P_{\cF}(D))}$, so if we assume $C(\cF+1)\geq C_{3}(C(\cF))$ and $P_{\cF+1}(D)\geq T_{C(\cF)}(P_{\cF}(D))$ we would conclude that $A^{c}$ is in $\Omega'_{\cF+1,D}$. 
    \item[(W4)] Let $A_1,\dots,A_{k}\subset\bbR^{\ell}$ with $A_{i}\in\Omega'_{\cF_i,D_i}$ and $\cF:=\max_{i}\{\cF_{i}\},\;D:=\sum_{i}D_i$. So $A_{i}\in\Omega_{C(\cF_{i}),P_{\cF_{i}}(D_{i})}$ and hence $\cup_{i}A_{i}$ is in $\Omega_{\max_{i}C(\cF_{i}),\sum_{i}P_{\cF_{i}}(D_{i})}$. To guarantee that $\cup{A_i}$ is in $\Omega'_{\cF,D}$, we need the assumption that $C(\cF)\geq\max_{i}C(\cF_{i})$ (which already follows from the assumption that $C$ is strictly monotone increasing), and that $P_{\cF}(D)\geq\sum_{i}P_{\cF_{i}}(D_{i})$. 
    \item[(W5)] Let $A_1,\dots,A_{k}\subset\bbR^{\ell}$ with $A_{i}\in\Omega'_{\cF_i,D_i}$ and $\cF:=\max_{i}\{\cF_{i}\},\;D:=\sum_{i}D_i$. So $A_{i}\in\Omega_{C(\cF_{i}),P_{\cF_{i}}(D_{i})}$ and hence $\cap A_{i}$ is in $\Omega_{C_{5}(\max_{i}C(\cF_{i})),Q_{\cF}(\sum_{i}P_{\cF_{i}}(D_{i}))}$. So, to guarantee that $\cap A_{i}$ is in $\Omega'_{\cF+1,D}$ we need to assume $C(\cF+1)\geq C_{5}\left(\max_{i}C(\cF_{i})\right)$ and that $P_{\cF+1}(D)\geq Q_{\cF}(\sum_{i}P_{\cF_{i}}(D_{i}))$.
    \item[(W6)] Let $P\in\R[x_{1},\dots,x_{\ell}]$ with $d=\deg P$. Then $\{P=0\}$ is in $\Omega_{C_{6}(\ell),H_{\ell}(d)}$, so to guarantee that $\{P=0\}$ is in $\Omega'_{\ell,d}$ we need to assume that $C\geq C_{6}$ and $P_{\ell}(D)\geq H_{\cF}(D)$.
\end{itemize} 
It is clear that $C,P_{\cF}$ can be constructed inductively so that they satisfy all of the conditions outlined above. Moreover, if $(\cS,\Omega)$ is effective in the sense that $C_{3},C_{5},C_{6},P_{\cF},Q_{\cF},H_{\ell}$ are given primitive recursive function, then $C,P_{\cF}$ can be chosen to be primitive recursive functions as well, and hence the reduction to $\Omega'$ is effective. 
\end{proof} 
%As mentioned above, the proof of \Cref{lem:format_normalized} in the W\so-minimal and P\so-minimal case is somewhat more involved. It requires a notion of \say{weak} and \say{pre} $\Omega$-format for structure trees, defined similarly to $\Omega$-format (almost verbatim in fact, one only needs to change the way format and degree are increased at every vertex), so that the format of the set associated to the root of a structure tree is the same as the pre $\Omega$-format or weak $\Omega$-format.  The same conclusion holds - the weak$\backslash$ pre $\Omega$-format of sets in $\Omega'_{\cF,D}$ is bounded by
%\begin{equation}
%    \underset{[\cF]\to\{C\}}{\max}\{C_{\cF}(\dots(C_{1}(\cF))\dots)\},
%\end{equation}
%where the set $\{C\}$ is the set all functions $C(\cF)$ replacing appearances of $\cF+1$ in the axioms that $\Omega$ satisfies, and the maximum is taken over all sequences $\left(C_{1},\dots,C_{\cF}\right)$ of length $\cF$ in this set.
\subsection{Proof of \Cref{prop:main_result}}
Let $(\cS,\Omega)$ be P\so-minimal. We claim that the pair $(\cS,\Omega^{*})$ satisfies axioms (*1)-(*6) in the notation of \Cref{lem:format_normalized}. Indeed, (*1),(*2),(*4) follow directly from the definition of $\Omega^{*}$. The axiom (*5) follows from \Cref{thm:*omega_effective_cellular_decomposition}, indeed, the intersection of a collection of definable sets is the union of certain cells from a cellular decomposition compatible with the collection. Axiom (*3) also follows directly from the definition of $\Omega^{*}$, except for the part about $A^{c}$, which again follows from \Cref{thm:*omega_effective_cellular_decomposition}. Finally, axiom (*6) readily follows from the fact that $\Omega\leq\Omega^{*}$.

 By \Cref{lem:format_normalized} we can define an FD-filtration $\Omega'$ which is equivalent to $\Omega^{*}$, and such that $(\cS,\Omega')$ is W\so-minimal. Moreover, $(\cS,\Omega')$ has \s CD since $\Omega'$ is equivalent to $\Omega^{*}$. We finish by applying \Cref{prop:Wsharp_SCD_isSharp}. \qedsymbol
\section{Sharp Triangulation}
\label{sec:Sharp_Triangulation}
In \cite{Pfaffian_cells}, a sharp triangulation theorem for $\rpfaff$ is deduced from \s CD and the ordinary proof of triangulation in o-minimality as it appears in \cite{Intro_o_min}. The deduction is simple; one needs to verify that the formulas describing the operation in \cite{Intro_o_min} have format $\format$ and degree $\degree$. Thus, we have the following. 
\begin{thm}
\label{thm:sharp_triangulation}
Let $(\cS,\Omega)$ be a \so-minimal stucture with \s CD. Let $Y\subset I^{\ell}$ be closed and definable, $X_{1},\dots,X_{k}\subset Y$ definable subsets such that all these sets have format $\cF$ and degree $D$. Then there exists a simplicial complex $K$ of size $\degree$ with vertices in $\Q^{\ell}$ and a definable homeomorphism $\Phi:|K|\to Y$ of format $\format$ and degree $\degree$, such that each $X_{i}$ is a union of images of simplices. 
\end{thm}
We immediately conclude the following. Note crucially, that \s CD is not needed in this conclusion. Rather, we use the fact that $\Omega^{*}$ has \s CD and that $\Omega\leq\Omega^{*}$.
\begin{thm}[Bound on sum of Betti numbers]
Let $\left(\cS,\Omega\right)$ be a presharp structure. If $X\in\Omega_{\cF,D}$ is compact, then the sum of the Betti numbers of $X$ is bounded by $\degree$.
\end{thm}
\begin{proof}
Assume without loss of generality that $X\subset[0,1]^{\ell}$. Thus $X$ is in $\Omega^{*}_{\format,\degree}$, and since $\Omega^{*}$ is equivalent to a \so-minimal filtration with \s CD, we are in position to apply \Cref{thm:sharp_triangulation} above, with $Y=[0,1]^{\ell}$ and $X_{1}=X,\;k=1$, obtaining a definable homeomorphism $\Phi:|K|\to[0,1]^{\ell}$ that lies in $\Omega^{*}_{\format,\degree}$. In particular, the inverse image $\Phi^{-1}(X)$, is a semialgebraic set that is the union of at most $\degree$ simplices, and therefore by \cite[Theorem 1]{semiaglerbaic_betti_numbers} the sum of the Betti numbers of $\Phi^{-1}(X)$ is bounded by $\degree$, and so we have the same bound for the sum of the Betti numbers of $X$ as well. 
\end{proof}
\bibliographystyle{plain}
\bibliography{References}
%\printbibliography
\end{document}